\documentclass[10pt,reqno]{amsart}

\usepackage[a4paper]{geometry}
\usepackage{enumerate,enumitem}
\usepackage[utf8]{inputenc}
\usepackage[T1]{fontenc} 
\usepackage[english]{babel}
\usepackage{array}
\usepackage{amsmath,amssymb,amsthm}
\usepackage{xcolor}
\usepackage{graphicx}
\usepackage{mathrsfs}
\usepackage{url}
\usepackage{hyperref}
\hypersetup{colorlinks=true,breaklinks=true,urlcolor=blue,linkcolor=blue,bookmarksopen=true} 
\allowdisplaybreaks

\newcommand{\be} {\begin{equation}}
\newcommand{\ee} {\end{equation}}
\newcommand{\bea} {\begin{eqnarray}}
\newcommand{\eea} {\end{eqnarray}}
\newcommand{\Bea} {\begin{eqnarray*}}
\newcommand{\Eea} {\end{eqnarray*}}

\newcommand*{\X}{\mathcal{X}}
\newcommand*{\M}{\mathcal{M}}
\newcommand*{\A}{\mathcal A}
\renewcommand{\P}{\mathcal P}

\newcommand{\1}{\mathbf 1}
\newcommand*{\N}{\mathbb{N}}
\newcommand*{\B}{\mathcal{B}}
\newcommand*{\Z}{\mathbb{Z}}

\newcommand*{\R}{\mathbb{R}}
\newcommand*{\ro}{r}

\newcommand*{\e}{\mathrm{e}}

\newcommand*{\E}{\mathbb{E}}

\newcommand{\dis}{\displaystyle}

\newtheorem{theo}{Theorem}[section]

\newtheorem{defi}[theo]{Definition}
\newtheorem{coro}[theo]{Corollary}
\newtheorem{lem}[theo]{Lemma}

\newtheorem{hypo}[theo]{Assumption}
\newtheorem{Rq}[theo]{Remark}
\numberwithin{equation}{section}

\definecolor{darkred}{rgb}{0.9,0.1,0.1}

\begin{document}

\title[Ergodic behavior of semigroups via Doeblin's conditions]{Ergodic behavior of non-conservative  semigroups via generalized Doeblin's conditions}

\author{
	Vincent \textsc{Bansaye}
		\and
	Bertrand \textsc{Cloez}
		\and
	Pierre \textsc{Gabriel}
}
\def\runauthor{
	Vincent \textsc{Bansaye}, Bertrand \textsc{Cloez}, Pierre \textsc{Gabriel}
}

\date{}
    \address[V. \textsc{Bansaye}]{CMAP, \'Ecole Polytechnique, Route de Saclay, 91128 Palaiseau cedex, France.}
    \email{vincent.bansaye@polytechnique.edu}
    \address[B. \textsc{Cloez}]{MISTEA, INRA, Montpellier SupAgro, Univ. Montpellier, 2 place Pierre Viala, 34060 Montpellier, France.}
    \email{bertrand.cloez@inra.fr}
    \address[P. \textsc{Gabriel}]{Laboratoire de Math\'ematiques de Versailles, UVSQ, CNRS, Universit\'e Paris-Saclay,  45 Avenue des \'Etats-Unis, 78035 Versailles cedex, France.}
    \email{pierre.gabriel@uvsq.fr}

\begin{abstract}

We provide   quantitative estimates  in total variation distance for positive semigroups, which can be non-conservative and  non-homogeneous. The techniques relies on a family of conservative semigroups that describes  a typical particle   and Doeblin's type conditions inherited from \cite{CV14} for coupling the associated process. Our aim is to provide quantitative estimates for linear partial differential equations and we develop several applications for population dynamics in varying environment. We start with the asymptotic profile for a growth  diffusion model  with time and space non-homogeneity. Moreover we provide general estimates for semigroups which become asymptotically homogeneous, which are  applied to an age-structured population model. Finally, we obtain a speed of convergence for periodic semigroups and new bounds in the homogeneous setting. They are illustrated on the renewal equation.

\tableofcontents

\end{abstract}

\keywords{Positive semigroups; non-autonomous linear evolution equations; measure solutions; ergodicity; Krein-Rutman theorem; Floquet theory; branching processes; population dynamics}

\subjclass[2010]{Primary 35B40; Secondary 47A35, 47D06, 60J80, 92D25}

\maketitle

\section{Introduction}
The solutions of the Cauchy problem associated to a linear Partial Differential Equation (PDE) can be expressed through a semigroup of linear operators. 
In the present work, we are interested in the ergodic properties of positive  semigroups  $(M_{s,t})_{t\geq s\geq 0}$ acting on measures,
and their application to the  study of the asymptotic profile of populations evolving in varying environment, which can be described by linear (nonautonomous) PDEs. 
Roughly speaking,  for any $t\geq s\geq 0,$ $M_{s,t}$ is both a positive linear operator on a space of measures ($\mu\mapsto\mu M_{s,t}$)
and on a space of measurable functions ($f\mapsto M_{s,t}f$),
and the family $(M_{s,t})_{t\geq s\geq 0}$ satisfies the semigroup  property 
$$ 
\forall s\leq u \leq t, \qquad M_{s,t}=M_{s,u}M_{u,t}.
$$
For a measure $\mu$ and a measurable function $f$, we  denote by $\mu(f)$ the integral of $f$ against $\mu.$
We  establish  ergodic approximations of the following form
$$\mu M_{s,t}  \approx \mu(h_s)\, r_{s,t}\,   \gamma_{t}$$
when $t\rightarrow \infty,$ for a fixed initial time $s.$
The first term in this long-time decomposition is a linear form $\mu\mapsto\mu(h_s)$ on the space of measures, independent of $t,$ which provides the long term impact of the initial distribution $\mu$ through the function $h_s.$
The second term is a family $(r_{s,t})_{t\geq s}$ of positive real numbers, independent of $\mu,$ describing the evolution of the ``mass''.
Finally $\gamma_t$ is the asymptotic probability distribution, which does not depend on $s$ nor $\mu.$
The harmonic function $h_s$ is unique up to normalization, but the two families $(r_{s,t})_{t\geq s}$ and $(\gamma_t)_{t\geq0}$ are not.
Nevertheless in particular situations, they can be chosen in certain relevant classes in which they are unique.
In Section~\ref{applis} we detail and illustrate such cases, briefly presented here:

\paragraph{\it Homogeneous semigroups}
In the homogeneous setting $M_{s,t}=M_{t-s},$ and provided a topology on the space of measures,  spectral theorems  suggest the  behavior
$$\mu M_{t}  = \mu(h)  \e^{\lambda t}   \gamma + \mathcal O(\e^{(\lambda-\varepsilon )t}),$$
where $\lambda$ is the dominant eigenvalue of the infinitesimal generator of the semigroup, $\gamma$ and $h$ are the associated eigenvectors, and $\varepsilon$ is the spectral gap.
This is an immediate consequence of the Perron Frobenius Theorem \cite{Perron,Frobenius} in finite state space setting.
In a general Banach lattice the existence of the eigentriplet $(\lambda,\gamma,h)$ is ensured by the Krein-Rutman Theorem \cite{Krein-Rutman} when the semigroup (or the resolvent of its generator) is positive, irreducible, and compact.
A refined variant of the Krein-Rutman theorem, with spectral gap, is proved in~\cite{MS16} in the setting of a Banach lattice of functions.
The proof relies on a spectral analysis 
and applies to positive semigroups with a generator which satisfies a strong maximum principle
and admits a decomposition verifying a power compactness condition.
In contrast with these approaches, our method is based on a contraction argument and can be efficiently applied to time-inhomogeneous semigroups.

\paragraph{\it Asymptotically homogeneous semigroups}
In the case where there exists a homogeneous semigroup $(N_t)_{t\geq0}$ such that $M_{s,s+t}\approx N_t$ for $s$ large, we prove that the principal eigenvector $\gamma$ of $(N_t)_{t\geq0}$ provides a stationary asymptotic profile
$$\mu M_{s,t}  \approx \mu(h_s)\, r_{s,t}\,   \gamma.$$
But $r_{s,t}$ is not necessarily an exponential growth provided by the associated eigenvalue,
and $h_s$ is not the associated eigenfunction.

\paragraph{\it Periodic semigroups}
When there exists $T>0$ such that $M_{s+T,t+T}=M_{s,t}$ for all $s\leq t,$ the semigroup is said to be periodic.
In this case we can choose for $(\gamma_{t})_{t\geq0}$ a $T$-periodic family, and similarly as in the homogeneous case, the evolution of the mass is exponential.
More precisely there exist a real number $\lambda_F,$ named Floquet eigenvalue after the work of G. Floquet~\cite{F83}, and a periodic family $(\eta_{t})_{t\geq0}$ bounded from above and below such that
$$\mu M_{s,t}  \approx \mu(h_s)\, \e^{\lambda_Ft+\eta_t}\, \gamma_{t}.$$

In all cases, the  bound for the speed of convergence  is expressed in the total variation norm (see Section~\ref{ssect:preliminaries} for the definition),
which is the natural distance for coupling processes in probability.
The proof relies on an auxiliary conservative semigroup $P^{(t)}$, defined for every bounded function $f$ and any times $0\leq s \leq u \leq t $ by
$$
P^{(t)}_{s,u}f=\frac{M_{s,u}(fm_{u,t})}{m_{s,t}}, \qquad \text{where }\ m_{s,t}=M_{s,t} \1,
$$
for which ergodic behaviour can be obtained through coupling arguments. This auxiliary semigroup describes the trajectory of a typical particle and has been used recently 
for the study of branching Markov processes in discrete and continuous time \cite{BDMT, BH15, B15, M16} and  processes  killed at a boundary  \cite{CV14, DelV, MMV12}. We come back
in Appendix~\ref{appendix:proba} on the link between these topics  in probability and  ergodic estimates for semigroups. \\
Doeblin and Lyapounov techniques (or petite sets) \cite{D40,MT} provide  then a powerful tool to control the ergodic behavior of this auxiliary Markov process. More generally,  the constructions of
auxiliary Markov processes derived from a typical or tagged particle  have been well developed  in probability and play a key role in the asymptotic study of stochastic
processes. They appear in  Feynman-Kac formula \cite{DelM04} and in spine technics via  many-to-one formulae \cite{HR17} for  the probabilistic study of branching processes \cite{BW17, C17, EHK10} and
fragmentation processes \cite{B06}, to name but a few. \\
When working on a compact state space or benefiting from an atom or a compact set  uniformly accessible for the whole state space, one can hope to check  Doeblin conditions on the auxiliary semigroup. Recall 
that a conservative, positive and homogeneous semigroup $(Q_t)_{t\geq 0}$ satisfies the Doeblin condition if there exist a constant $c>0$, a coupling probability measure $\nu$ and a time $t_0>0$ such that for all positive and bounded function $f$,
$$
Q_{t_0} f \geq c\, \nu(f).
$$
This condition is equivalent to a contraction in total variation distance and then provides a convenient tool of analysis for non-homogenous models.
Sharp assumptions expressed in function of $M$ have recently been obtained in \cite{CV14} to get a Doeblin condition for  the auxiliary semigroup in a context of absorbed Markov process. These conditions 
are weaker than the classical conditions using Birkhoff contraction \cite{B57,Nussbaum,GaubertQu} and equivalent to uniform exponential convergence.\\
In Lemma~\ref{denis} we prove that Doeblin's condition hold for the semigroup $P^{(t)},$ which in turn provides an  explicit  bound for the  decrease of
$$P^{(t)}_{s,t}f(x)-P^{(t)}_{s,t}f(y)=\frac{M_{s,t}f(x)}{m_{s,t}(x)}-\frac{M_{s,t}f(y)}{m_{s,t}(y)}$$
as $t\rightarrow \infty$ and the ergodic behavior  of the auxiliary semigroup.
The proof of this Lemma is essentially an adaptation of the method in~\cite{CV14,CV16} that we extend to general semigroups in non-homogeneous environment, while they restrict their study to absorbed Markov processes.
This more general semigroup setting allows us to capture a wider range of applications, like the renewal equation we consider in Section~\ref{applis}.
Moreover, we go beyond the contraction of the auxiliary semigroup $P^{(t)}$ and characterize the asymptotic behavior of $(M_{0,t})_{t\geq 0},$ which is a novelty compared to the previous results.\\
More precisely, for any initial time $s\geq0,$ we propose conditions involving a coupling probability measure $\nu$ which guarantee the existence of a positive bounded function $h_s$ and a family of probabilities $(\gamma_t)_{ t\geq 0}$ such that when $t\to\infty$
$$\sup_{\left\|\mu\right\|_{\mathrm{TV}}\leq1}\bigl\| \mu  M_{s,t}  -  \mu(h_s) \nu(m_{s,t}) \gamma_t   \bigr\|_{\mathrm{TV}} =o\big(\nu(m_{s,t})\big).$$
These conditions are stated in Section~\ref{sect:genstatandproof} and a quantified version of above convergence is proved. 
In Section~\ref{applis} the general result is declined in several applications, which are illustrated by concrete and intentionally simple examples of linear PDE issued from population dynamics
We  avoid too much technicality but provide some new estimates and explain the way assumptions can be checked. 
We first consider  in Section \ref{sect:diff}  a model of  population growing in a non-homogeneous and  diffusing in a varying environment, which is illustrated by ergodic random environment.
Intuitively, if the variation of parameters in the model is not vanishing in large times, one does not expect the convergence of $\gamma_t$.
In the case of homogeneous or asymptotically homogeneous semigroups,  we prove
that the asymptotic profile is given by a constant probability measure $\gamma$; see Section \ref{sect:homogene} and Section \ref{sect:quasih} respectively.
Finally, when the semigroup evolves periodically we prove that the asymptotic profile $\gamma_t$ is periodic; see Section \ref{sect:periodic}.
Results of Section~\ref{sect:homogene} (homogeneous semigroups), Section~\ref{sect:quasih} (asymptotically homogeneous semigroups) and Section~\ref{sect:periodic} (periodic semigroups) are illustrated on the renewal equation.
In these three  settings, we obtain new sharp conditions for convergence with explicit rate of convergence.

\section{General statement and proof}
\label{sect:genstatandproof}

\subsection{Preliminaries on measures and semigroups}\label{ssect:preliminaries}

We start by recalling some definitions and results about measure theory, and we refer to~\cite{Rudin} for more details and proofs.

Let $\X$ be a locally compact Hausdorff space and denote by $\B_b (\X)$ the space of bounded Borel functions $f:\X\to\R$ endowed with the supremum norm $\|f\|_\infty=\sup_{\X}|f|.$
We denote by $\M(\X)$ the space of regular signed Borel measures on $\X$\footnote{Notice that if $\X\subset \R^n$ is equipped with the induced topology, any signed Borel measure on $\X$ is regular.}, by $\M_+(\X)$ its positive cone ({\it i.e.} the set of regular finite positive Borel measures), and by $\P(\X)$ the subset of probability measures.
For two measures $\mu,\widetilde\mu\in\M(\X),$ we say that $\mu$ is larger than $\widetilde\mu,$ and write $\mu\geq\widetilde\mu,$ if $\mu-\widetilde\mu\in\M_+(\X).$
The Jordan decomposition theorem ensures that for any $\mu\in\M(\X)$ there exists a unique decomposition $\mu=\mu_+-\mu_-$
with $\mu_+$ and $\mu_-$ positive and mutually singular.
The positive measure $|\mu|=\mu_++\mu_-$ is called the total variation measure of the measure $\mu,$
and its mass is the total variation norm of $\mu$
\[\left\|\mu\right\|_{\mathrm{TV}}:=|\mu|(\X)=\mu_+(\X)+\mu_-(\X).\]
Clearly we have the identity\footnote{We see here that the definition we use for the total variation norm differs from the usual probabilistic definition of a factor $1/2.$}
\[\left\|\mu\right\|_{\mathrm{TV}}=\sup_{\|f\|_\infty\leq1}|\mu (f)|,\]
where the supremum is taken over measurable functions.
By virtue of the Riesz representation theorem, this supremum can be restricted to the continuous functions vanishing at infinity\footnote{A function $f$ on a locally compact Hausdorff space $\X$ is said to vanish at infinity if to every $\varepsilon>0,$ there exists a compact set $K\subset\X$ such that $|f(x)|<\varepsilon$ for all $x\in\X\setminus K.$},
{\it i.e.} $f\in C_0(\X)=\overline{C_c(\X)}.$
The Riesz representation theorem also ensures that $(\M(\X),\left\|\cdot\right\|_{\mathrm{TV}})$ is a Banach space, as a topological dual space.
It is worth noticing that the inequality $|\mu(f)|\leq \left\|\mu\right\|_{\mathrm{TV}}\|f\|_\infty$ which is valid for any $\mu\in\M(\X)$ and $f\in\B_b(\X)$
can be strengthened into $|\mu(f)|\leq \frac12\left\|\mu\right\|_{\mathrm{TV}}\|f\|_\infty$ when $\mu(\X)=0$ and $f\geq0.$

For any $\Omega\subset\X$ we denote by $\1_\Omega$ the indicator function of the subset $\Omega.$
And we denote by $\1$ the constant function equal to $1$ on $\X$, {\it i.e.} $\1=\1_\X$.

Now we turn to the definition of the (time-inhomogeneous) semigroups we are interested in.
Let $(\X_t)_{t\geq0}$ be a family of locally compact Hausdorff spaces.
A semigroup $M=(M_{s,t})_{0\leq s\leq t}$ is a family of linear operators defined as follows.
For any $t\geq s\geq 0,$ $M_{s,t}$ is a bounded linear operator from $\M(\X_s)$ to $\M(\X_t)$ through the left action
\[M_{s,t}:
\begin{array}{ccc}
\M(\X_s)&\to&\M(\X_t)\\
\mu&\mapsto& \mu M_{s,t}
\end{array},\]
and a bounded linear operator from $\B_b(\X_t)$ to $\B_b(\X_s)$ through the right action
\[M_{s,t}:
\begin{array}{ccc}
\B_b(\X_t)&\to&\B_b(\X_s)\\
f&\mapsto& M_{s,t}f
\end{array}.\]
The semigroup property means here that for all $s\leq u\leq t$ and $f\in \B_b(\X_t)$
\[M_{s,t}f=M_{s,u}(M_{u,t}f).\]
Moreover, we make the following assumptions.

\begin{hypo}
\label{as:Mst}
We assume that for all $t\geq s\geq 0$ we have
\begin{itemize}[itemsep=3mm]
\item[] $(\, f\in \B_b(\X_t), \  f\geq0\,)\ \implies\  M_{s,t}f\geq0,$\hfill(positivity)
\item[] $\forall x\in\X_s,\quad m_{s,t}(x):=(M_{s,t}\1)(x)>0,$\hfill(strong positivity)
\item[] $\forall (\mu,f)\in\M(\X_s)\times \B_b(\X_t),\qquad (\mu M_{s,t})(f)=\mu(M_{s,t}f).$\hfill(left-right compatibility)
\end{itemize}
\end{hypo}
Due to the compatibility condition, we can denote without ambiguity $\mu M_{s,t}f=(\mu M_{s,t})(f)=\mu(M_{s,t}f),$
and $(\mu,f)\mapsto\mu M_{s,t}f$ is a bilinear form on $\M(\X_s)\times \B_b(\X_t).$
Notice additionally that the compatibility condition allows to transfer the semigroup property and the positivity to the left action, {\it i.e.} for all $t\geq s\geq 0$, we have
\[\forall\, u\in[s,t],\ \forall\mu\in\M(\X_s),\qquad\mu M_{s,t}=(\mu M_{s,u})M_{u,t},\]
\[\mu\in\M_+(\X_s)\quad \implies\quad \mu M_{s,t}\in\M_+(\X_t).\]

\subsection{Coupling constants}

Let $\alpha,\beta>0$ and $\nu\in\P(\X_s).$

\begin{defi}[Admissible coupling constants]
\label{def:coupling_constants}
For any $N\geq 1$,
we say that $(c_{i},d_{i})_{1 \leq i \leq N} \in [0,1]^{2N}$   are  $(\alpha,\beta, \nu)$-admissible coupling constants for $M$  on $[s,t]$  if
there exist real numbers  $(t_i)_{ 0\leq i \leq N}$ satisfying
$s\leq t_0 \leq  \ldots \leq t_N \leq t$ and  probability  measures 
$\nu_{i}$ on $\X_{t_i}$ such that for all $i=1,\ldots, N$ and $x\in \X_{t_{i-1}}$,
\begin{equation}\label{as:A1I}\tag{A1}
 \delta_x M_{t_{i-1},t_{i}}\geq c_{i}m_{t_{i-1},t_{i}}(x)\nu_{i},
\end{equation}
and for all $i\in \{1,\ldots, N-1\}$, $\tau\geq t_N$ and $x\in\X_{t_i}$,
\begin{equation}\label{as:A2I}\tag{A2}
 d_{i} m_{t_{i},\tau}(x) \leq \nu_{i} (m_{t_{i},\tau})
\end{equation}
and  for all $\tau\geq t_N$ and $x\in\X_{t_N}$,
\begin{equation}\label{as:A3I}\tag{A3}
 \quad  m_{t_{N},\tau}(x) \leq     \alpha  \, c_{N} \, \nu_{N} (m_{t_{N},\tau})   
\end{equation}
and for all $\tau\geq t_N$ and $x\in\X_{s}$, 
\begin{equation}\label{as:A4I}\tag{A4}
 m_{s,\tau}(x) \leq \beta \,  \nu(m_{s,\tau}).
 \end{equation}
\end{defi}

In the conservative case,  Assumption~\eqref{as:A1I} is the classical Doeblin assumption. It is a strong  irreducibility property:  whatever the initial distribution is, the semigroup  between the times $t_{i-1}$ and $t_i$ is  lowerbounded by a fixed measure $\nu_i$. This condition is  then sufficient (and even necessary) for uniform exponential convergence. \\
But this condition is no longer sufficient for non-conservative semi-group. First, the mass of the process has to be added in  ~\eqref{as:A1I}, as will be seen in examples when the mass vanishes. Moreover  The mass of the semi-group has to be essentially the same for any starting distribution. This is the meaning of Assumptions~\eqref{as:A2I}, \eqref{as:A3I} and \eqref{as:A4I}.

The two first assumptions  allow to get the contraction of the auxiliary semigroup $P^{(t)}$ in the total variation norm following \cite{CV14,CV16}, see Lemma \ref{denis}. 
The two additional assumptions are needed
to prove  the existence of harmonic-type functions and control the speed of convergence in the general result, see  forthcoming Lemma \ref{alacon}. 
Assumptions  \eqref{as:A2I}, \eqref{as:A3I}, \eqref{as:A4I}   all involve the control of the mass $m$ for large times and  will be proved in the same time by a coupling argument in applications of Section \ref{applis}.   The associated constants may change in varying environment, see Section \ref{sect:diff}.

We denote by  $\mathcal H_{\alpha, \beta, \nu}(s,t)$  the set  of $(\alpha,\beta, \nu)$-admissible coupling constants $(c_{i},d_i)_{1 \leq i \leq N}$ for $M$ on $[s,t].$
It can be easily seen from Definition~\ref{def:coupling_constants} that for this set to be nonempty, the constants $\alpha$ and $\beta$ have to be at least greater than or equal to $1.$
We are interested in the optimal admissible coupling and  we set
\begin{equation}
\label{eq:coupling-capacity}
C_{\alpha, \beta, \nu }(s,t)=\sup_{ \mathcal H_{\alpha, \beta, \nu}(s,t)} \left\{- \sum_{i=1}^{N} \log(1-c_{i}d_{i}) \right\},
\end{equation}
where by convention $\sup\varnothing =0$. We observe that $t\mapsto C_{\alpha,\beta, \nu}(s,t)$ is positive and non-decreasing.

\subsection{General result}

Here we state the general result we obtain about the ergodicity of semigroups $M$ which satisfy Assumption~\ref{as:Mst}.
\begin{theo}
\label{th:main}
Let  $ s \geq 0$ and assume that there exist $\alpha, \beta\geq 1,$ and $\nu$ a probability measure on $\X_s$ such that $C_{\alpha, \beta, \nu}(s,t)\rightarrow \infty$ as $t\rightarrow \infty$.
Then there exists a unique function $h_s : \X_s\rightarrow [0,\infty) $ such that
for any $\mu\in \mathcal M(\X_s),$ $\gamma\in \mathcal M(\X_{s_0})$ with $s_0 \in [0, s],$
and for any $t$ such that $C_{\alpha,\beta,\nu}(s,t)\geq \log(4\alpha)$, 
$$\biggl\| \mu  M_{s,t}  -  \mu(h_s) \nu(m_{s,t})  \frac{ \gamma M_{s_0,t} }{ \gamma (m_{s_0,t})}  \biggr\|_{\mathrm{TV}} \leq  8(2+\alpha)  |\mu|(h_s)\,\nu(m_{s,t}) \, \e^{-C_{\alpha,\beta, \nu}(s,t)}.$$
Moreover $h_s(x) \in (0,\beta]$ for any $x\in \X_s$ and $\nu(h_s)=1$.
\end{theo}

Before the proof, let us make two remarks. First,
under the assumption of Theorem~\ref{th:main}, we also prove that for all $t\geq s$,
\[\biggl\| \mu  M_{s,t}  -  \mu(h_s) \nu(m_{s,t})  \frac{ \gamma M_{s_0,t} }{ \gamma (m_{s_0,t})}  \biggr\|_{\mathrm{TV}} \leq  2(2+\alpha)\beta  \|\mu\|_{\mathrm{TV}}\,\nu(m_{s,t}) \, \e^{-C_{\alpha,\beta, \nu}(s,t)}.\]
This bound is thus valid for any time. Second, we can change the measure $\nu$ as follows.

\begin{Rq} \label{rq:A4}
Suppose that for all $\mu\in\M(\X_s)$ the function $t\mapsto\mu(m_{s,t})$ is continuous and the Assumptions of Theorem \ref{th:main} hold.
Then for all $\widetilde\nu\in\P(\X_s)$ there exists a constant $\widetilde\beta$ such that \eqref{as:A4I} is still valid if we replace $\nu$ and $\beta$ by $\widetilde\nu$ and $\widetilde\beta.$
Indeed Theorem~\ref{th:main} applied to $\mu=\widetilde\nu$
ensures that $\widetilde{\nu}(m_{s,t})/\nu(m_{s,t})\to\widetilde{\nu}(h_s)>0$ when $t\to\infty.$
Since $t\mapsto\nu(m_{s,t})/\widetilde\nu(m_{s,t})$ is continuous, it is bounded on $[s,+\infty).$
Using \eqref{as:A4I}, we deduce that for all $T\geq t_N$,
\[\|m_{s,T}\|_\infty\leq\beta\nu(m_{s,T})\leq\beta\sup_{t\geq s}\left(\frac{\nu(m_{s,t})}{\widetilde\nu(m_{s,t})}\right)\widetilde\nu(m_{s,T})=\widetilde\beta\,\widetilde\nu(m_{s,T}).\]
\end{Rq}

\subsection{Proof of Theorem \ref{th:main}}

We recall from the introduction the definition of $P^{(t)}.$
For any $t\geq u\geq s\geq 0,$ the linear operator $P^{(t)}_{s,u}:\B_b(\X_u)\to\B_b(\X_s)$ is defined by
\[P^{(t)}_{s,u}f=\frac{M_{s,u}(fm_{u,t})}{m_{s,t}}.\]
By duality we define a left action $P^{(t)}_{s,u}:\M(\X_s)\to\M(\X_u)$ by
\begin{equation}
\label{defP_left}
\forall f\in \B_b(\X_u),\qquad(\mu P^{(t)}_{s,u})(f):=\mu(P^{(t)}_{s,u}f)=\int_{\X_s} \mu (dx) \frac{M_{s,u}(fm_{u,t})(x)}{m_{s,t}(x)}.
\end{equation}
We  recall that this is a positive conservative semigroup.
Indeed we readily check that $P^{(t)}_{s,u}\,\1=\1$ and $P^{(t)}_{s,u}\,f\geq 0$ if $f\geq 0.$ 
Moreover 
$$
P^{(t)}_{s,u}(P^{(t)}_{u,v}f)=\frac{M_{s,u}\big((P^{(t)}_{u,v}f)m_{u,t}\big)}{m_{s,t}}\\
=\frac{M_{s,u}\Big(\frac{M_{u,v}(fm_{v,t})}{m_{u,t}}m_{u,t}\Big)}{m_{s,t}}
=P^{(t)}_{s,v}f.$$
It is also worth noticing that for all $t\geq s\geq0$ and all $x\in\X_s$
\begin{equation}\label{Pstt}
\delta_xP_{s,t}^{(t)}=\frac{\delta_xM_{s,t}}{m_{s,t}(x)}.
\end{equation}

The first key ingredient  is the following lemma, which gives
the ergodic behavior of the
auxiliary conservative  semigroup under  assumptions \eqref{as:A1I} and \eqref{as:A2I}.
This is an almost direct  generalization  of \cite{CV14}, which holds 
for homogeneous and sub-conservative (or sub-Markov) semigroups; namely $M_{s,t}=M_{t-s}$ and $M_t \mathbf{1} \leq \mathbf{1}$, for all $t\geq s\geq 0$. These semigroups are
associated to the evolution of absorbed (or killed) Markov processes (see Section \ref{appendix:proba}). This is also related to \cite[Chapter 12]{DelMoral2013} or \cite[Chapter 4.3.2]{DelM04}. The proof is given here for the sake of completeness. 

\begin{lem}[Doeblin contraction]
\label{denis}
Let  $0\leq s\leq t$ and $(c_i,d_i)_{1\leq i\leq N}$ satisfying  \eqref{as:A1I} and \eqref{as:A2I}
for the time subdivision $s\leq t_0\leq \ldots  \leq  t_{N} \leq t$. Let $\tau \geq t_N$.

(i) For any  $i=1, \ldots,  N$, there exists $\mu_i\in\P(\X_{t_i})$ such that for all $x\in\X_{t_{i-1}}$
\[
\delta_xP_{t_{i-1},t_{i}}^{(\tau)}\geq c_i d_i \mu_i.\]

(ii)  For any $\mu,\widetilde{\mu} $ finite measures on  $\X_s$,
\[\left\|\mu P_{s,\tau}^{(\tau)}-\widetilde{\mu} P_{s,\tau}^{(\tau)}\right\|_{\mathrm{TV}}\leq \prod_{i\leq N} (1-c_id_i)\bigl\|\mu -\widetilde{\mu}\bigr\|_{\mathrm{TV}}.\]

(iii) For any non-zero $\mu, \widetilde{\mu} \in\M_+(\X_s)$,
\begin{align*}
\biggl\|\frac{\mu M_{s,\tau}}{\mu (m_{s,\tau})}-\frac{\widetilde{\mu} M_{s,\tau}}{\widetilde{\mu}  (m_{s,\tau})}\biggr\|_{\mathrm{TV}}
&\leq 2\prod_{i\leq N} (1-c_i d_i).
\end{align*}
\end{lem}

\begin{Rq}[Sharper bound]
In view of the proof below, one can replace  Lemma \ref{denis} (iii) by 
\begin{equation}
\label{eq:Wass}
\biggl\|\frac{\mu M_{s,\tau}}{\mu (m_{s,\tau})}-\frac{\widetilde{\mu} M_{s,\tau}}{\widetilde{\mu}(m_{s,\tau})}\biggr\|_{\mathrm{TV}}
\leq 2 \prod_{i\leq N} (1-c_id_i) \, \mathcal{W}_{s,t_N} (\mu,\widetilde{\mu}),
\end{equation}
where $\mathcal{W}_{s,t_N}$  is a  Wasserstein distance (see for instance \cite{Vil09}) defined by
$$
\mathcal{W}_{s,t_N} (\mu,\widetilde{\mu}) =\inf_{ \Pi} \frac{1}{\mu(m_{s,t_N}) \widetilde{\mu}(m_{s,t_N})} \int_{\X_s} m_{s,t_N}(y) m_{s,t_N}(x) \mathbf{1}_{x\neq y} \Pi(dx,dy).$$
and  the infimum runs over all coupling measures $\Pi$ of $\mu$ and $\widetilde{\mu}$; a coupling measure is a positive measure on $\X_s^2$ whose marginals are given by $\mu$ and $\widetilde{\mu}$. Even if the right-hand side of \eqref{eq:Wass} vanishes now when $\mu= \tilde{\mu}$, this bound depends on $t_N$ and incalculable quantities. However if there exists $A_s,B_s>0 $ such that $A_s \leq \sup_{\tau \geq s} \nu m_{s,\tau}/\mu m_{s,\tau}\leq B_s$ then Equation \eqref{eq:Wass} entails that
\begin{equation*}
\biggl\|\frac{\mu M_{s,\tau}}{\mu (m_{s,\tau})}-\frac{\widetilde{\mu} M_{s,\tau}}{\widetilde{\mu}(m_{s,\tau})}\biggr\|_{\mathrm{TV}}
\leq 2 \frac{B^2_s}{A^2_s}\prod_{i\leq N} (1-c_id_i) \, \Vert \mu - \widetilde{\mu} \Vert_{\mathrm{TV}}.
\end{equation*}
See Section \ref{sect:diff} and Inequality \eqref{eq:encad-m} for an example.
\end{Rq}

\begin{proof}[Proof of Lemma \ref{denis}]
\textit{Proof of (i).}
 Let $i\leq N$ and $f$ be a positive function of $\B_b(\X_{t_i})$. Using
\eqref{as:A1I}, we have
\be
\label{do1}
\delta_{x}M_{t_{i-1},t_{i}}(fm_{t_{i},\tau})\geq c_i \nu_i (fm_{t_{i},\tau})m_{t_{i-1},t_{i}}(x)=c_i\nu_i(fm_{t_{i},\tau})\frac{m_{t_{i-1},t_{i}}(x)}{m_{t_{i-1},\tau}(x)}m_{t_{i-1},\tau}(x).
\ee
Let us find $\mu_i$ satisfying
\be
\label{do2}
\nu_i(fm_{t_{i},\tau})\frac{m_{t_{i-1},t_{i}}(x)}{m_{t_{i-1},\tau}(x)}\geq d_i\mu_i(f).
\ee
Using  \eqref{as:A2I}, the  semigroup and positivity properties ensure that
$$d_i m_{t_{i-1},\tau}(x)=d_i \delta_x M_{t_{i-1},t_{i}}(m_{t_{i},\tau})\leq m_{t_{i-1},t_{i}}(x)
\nu_i(m_{t_{i},\tau}).$$
Thus
$$\nu_i(fm_{t_{i},\tau})\frac{m_{t_{i-1},t_{i}}(x)}{m_{t_{i-1},\tau}(x)}\geq d_i\mu_i(f),$$
where $\mu_i$ defined by
$$\mu_i(f)=\frac{\nu_i(fm_{t_{i},\tau})}{\nu_i(m_{t_{i},\tau})}$$
is a probability measure.
Recalling \eqref{defP_left}, $(i)$ follows from \eqref{do1} and \eqref{do2}.

\textit{Proof of (ii).}  We consider the  conservative linear operator $U_i$ on $\B_b(\X_{t_i})$, defined by
$$U_if(x):=\frac{P^{(\tau)}_{t_{i-1},t_i}f(x)-c_id_i\mu_{i}(f)}{1-c_id_i},$$
for $f\in \B_b(\X_{t_i})$ and $x\in\X_{t_{i-1}}$. It is positive by  $(i)$ and $\|U_i f\|_\infty\leq \|f\|_\infty$. Using
\begin{align*}
\delta_x P_{s,t_i}^{(\tau)}f- \delta_yP_{s,t_i}^{(\tau)}f 
&=\delta_xP^{(\tau)}_{s,t_{i-1}}P_{t_{i-1},t_i}^{(\tau)} f-\delta_yP^{(\tau)}_{s,t_{i-1}}P^{(\tau)}_{t_{i-1},t_i}f\\
&=(1-c_id_i)\bigl(\delta_x P^{(\tau)}_{s,t_{i-1}}(U_i f)-\delta_yP^{(\tau)}_{s,t_{i-1}}(U_i f)\bigr),
 \end{align*}
we get
  $$\|\delta_xP_{s,t_i}^{(\tau)}-\delta_yP_{s,t_i}^{(\tau)}\|_{\mathrm{TV}}\leq  (1-c_id_i)   \|\delta_xP_{s,t_{i-1}}^{(\tau)}-\delta_yP_{s,t_{i-1}}^{(\tau)} \|_{\mathrm{TV}}$$
since $\|U_i f\|_\infty\leq \|f\|_\infty$. Using that $P^{(\tau)}_{t_N,\tau}$ is also contraction since it is conservative, we obtain
 \be
 \|\delta_xP_{s,\tau }^{(\tau)}-\delta_yP_{s,\tau }^{(\tau)}\|_{\mathrm{TV}}\leq  2 \prod_{i\leq N} (1-c_id_i).
 \label{preborne}
 \ee
To conclude, we now check that for any conservative positive kernel $P$ on some $\X,$ any $\mu,\widetilde{\mu} \in \mathcal M(\X)$ such that $\mu(\X)=\widetilde{\mu}(\X)<\infty$,
\be
\label{dom}
\|\mu P-\widetilde{\mu} P\|_{\mathrm{TV}}
\leq\frac{1}{2} \sup_{x,y \in \X}\|\delta_xP-\delta_yP\|_{\mathrm{TV}}\|\mu-\widetilde{\mu}\|_{\mathrm{TV}} 
\leq \sup_{x,y \in \X}\|\delta_xP-\delta_yP\|_{\mathrm{TV}} .
\ee
Indeed, 
 $\mu P-\widetilde{\mu} P=(\mu-\widetilde{\mu})P= (\mu-\widetilde{\mu})_+P-(\widetilde{\mu}-\mu)_+P$
and $(\mu-\widetilde{\mu})_+(\X)=(\widetilde{\mu}-\mu)_+(\X),$ 
$$(\mu P-\widetilde{\mu}P)(f)=\frac1{(\mu-\widetilde{\mu})_+(\X)}\int_{\X^2}(\mu-\widetilde{\mu})_+(dx)(\widetilde{\mu}-\mu)_+(dy)\bigl(\delta_x Pf-\delta_y Pf\bigl),  $$
so  we get
$$
\|\mu P-\widetilde{\mu}P\|_{\mathrm{TV}}
\leq\sup_{x,y \in \X}\|\delta_xP-\delta_yP\|_{\mathrm{TV}}(\mu-\widetilde{\mu})_+(\X).$$
This  proves $\eqref{dom}$ since, by definition, $\|\mu-\widetilde{\mu}\|_{\mathrm{TV}}=(\mu-\widetilde{\mu})_+(\X) + (\widetilde{\mu}-\mu)_+(\X) = 2 (\mu-\widetilde{\mu})_+(\X)$ and yields $(ii)$. \\ 

\textit{Proof of (iii).}
Using now $\eqref{preborne}$ and recalling  (\ref{Pstt}), for any $x,y\in \X_s$, we have
\[\biggl\|\frac{\delta_xM_{s,\tau}}{m_{s,\tau}(x)}-\frac{\delta_yM_{s,\tau}}{m_{s,\tau}(y)}\biggr\|_{\mathrm{TV}}\leq 2 \prod_{i\leq N} (1-c_id_i).\]
Then for any nonzero $\mu\in\M_+(\X_s)$,
\begin{align*}
\biggl\|\frac{\mu M_{s,\tau}}{\mu (m_{s,\tau})}-\frac{\delta_yM_{s,\tau}}{m_{s,\tau}(y)}\biggr\|_{\mathrm{TV}}&=\frac{1}{\mu m_{s,\tau}}\biggl\|\mu M_{s,\tau}-\frac{\mu m_{s,\tau}}{m_{s,\tau}(y)}\delta_yM_{s,\tau}\biggr\|_{\mathrm{TV}}\\
&\leq\frac{1}{\mu (m_{s,\tau})}\int_{\X_s} \mu(dx) \biggl\|\delta_x M_{s,\tau}-  \frac{m_{s,\tau}(x)}{m_{s,\tau}(y)}\delta_yM_{s,\tau}\biggr\|_{\mathrm{TV}}\\
&=\frac{1}{\mu (m_{s,\tau})} \int_{\X_s} {\mu(dx)} m_{s,\tau}(x)\biggl\|\frac{\delta_x M_{s,\tau}}{m_{s,\tau}(x)}-  \frac{\delta_yM_{s,\tau}}{m_{s,\tau}(y)}\biggr\|_{\mathrm{TV}}\ 
\leq2 \prod_{i\leq N} (1-c_id_i).
\end{align*}
The inequality can  be similarly  extended  from $\delta_y$  to a finite measure  $\nu$, which proves $(iii)$.
\end{proof}

Now $\eqref{as:A3I} $  is involved to get the following non-degenerate bound for the mass.

\begin{lem} 
\label{alacon} Let $0\leq s\leq t$ and $(c_i,d_i)_{1\leq i\leq N}$ satisfying  \eqref{as:A1I},  \eqref{as:A2I}  and  \eqref{as:A3I} 
for the time subdivision $s\leq t_0\leq \ldots  \leq  t_{N}  \leq t$. For any $\tau\geq t_N$ and any measure $\mu \in \M_+(\X_s)$, we have
\be
\label{borne}
\frac{\mu M_{s,t_{N}}}{\mu m_{s,t_{N}}}\left(\frac{m_{t_{N},\tau}}{\| m_{t_{N},\tau} \|_{\infty}}\right) \geq \frac{1}{ \alpha}
\ee
and for any $x\in \X_s$,
\be
\label{viveineg}
\left\vert \frac{m_{s,\tau}(x)}{\mu(m_{s,\tau})}-\frac{ m_{s,t}(x)}{\mu(m_{s,t})}\right\vert \leq 2\alpha \frac{ m_{s,t_N}(x)}{\mu(m_{s,t_N})} \prod_{i\leq N} (1-c_id_i) .
\ee
If furthermore $2\alpha \prod_{i\leq N} (1-c_id_i) <1$, then  for any $x\in \X_s$,
\be
\label{onemore}
\left\vert \frac{m_{s,\tau}(x)}{\mu(m_{s,\tau})}-\frac{ m_{s,t}(x)}{\mu(m_{s,t})}\right\vert \leq  \frac{ m_{s,t}(x)}{\mu(m_{s,t})}\frac{2 \alpha  \prod_{i\leq N} (1-c_id_i)}{1-\alpha \prod_{i\leq N} (1-c_id_i)}.
\ee
\end{lem}
\begin{proof}
First,  using \eqref{as:A1I},
$$\frac{\mu M_{s,t_N}}{\mu(m_{s,t_N})}= \frac{\mu M_{s,t_{N-1}}M_{t_{N-1},t_N}}{\mu(m_{s,t_N})}\geq c_N  \frac{\mu M_{s,t_{N-1}}(m_{t_{N-1},t_N})}{\mu(m_{s,t_N})}  \nu_N =  c_N   \nu_N.$$
Moreover   \eqref{as:A3I} ensures that for any $\tau\geq t_N$,
$$\|m_{t_N,\tau}\|_\infty \leq \alpha c_N \nu_N (m_{t_N,\tau}).$$
This proves $\eqref{borne}$. Now, the semigroup property yields 
$$
\frac{m_{s,\tau}(x)}{\mu(m_{s,\tau} )}=\frac{\delta_{x}M_{s,t_N}M_{t_N,\tau}\mathbf{1}}{\mu M_{s,t_N} M_{t_N,\tau}\mathbf{1}}
= \frac{ m_{s,t_N}(x)}{\mu(m_{s,t_N})}  \frac{\frac{\delta_x M_{s,t_N}}{ m_{s,t_N}(x)}   (m_{t_N,\tau})}{ \frac{\mu M_{s,t_N}}{\mu(m_{s,t_N})}(m_{t_N,\tau})}.$$
Then
$$
\frac{m_{s,\tau}(x)}{\mu(m_{s,\tau})}-\frac{ m_{s,t_N}(x)}{\mu(m_{s,t_N})}= \frac{ m_{s,t_N}(x)}{\mu(m_{s,t_N})} 
\frac{\left[\frac{\delta_x M_{s,t_N}}{ m_{s,t_N}(x)}-\frac{\mu M_{s,t_N}}{ \mu(m_{s,t_N})}\right]    (m_{t_N,\tau})}{ \frac{\mu M_{s,t_N}}{\mu (m_{s,t_N})}(m_{t_N,\tau})}.
 $$
Dividing by $\| m_{t_N,\tau} \|_{\infty}$ and using  $m_{t_N,\tau}\geq 0$ and recalling that $\gamma(\X)=0$ and $f\geq0$ implies that $|\gamma(f)|\leq \frac12\left\|\gamma\right\|_{\mathrm{TV}}\|f\|_\infty$, we get
$$\left\vert \frac{m_{s,\tau}(x)}{\mu (m_{s,\tau})}-\frac{ m_{s,t_N}(x)}{\mu(m_{s,t_N})}\right\vert \leq  \frac{ m_{s,t_N}(x)}{\mu(m_{s,t_N})}
\frac{\frac{1}{2}\Big\|\frac{\delta_x M_{s,t_N}}{ m_{s,t_N}(x)}-\frac{\mu M_{s,t_N}}{\mu (m_{s,t_N})}\Big\|_{\mathrm{TV}}}{ \frac{\mu M_{s,t_N}}{\mu( m_{s,t_N})}\left(\frac{m_{t_N,\tau}}{\| m_{t_N,\tau} \|_{\infty}}\right)}.$$
Now combining Lemma \ref{denis} $(iii)$ and \eqref{borne} yields 
\be
\label{interm}
\left\vert \frac{m_{s,\tau}(x)}{\mu(m_{s,\tau})}-\frac{ m_{s,t_N}(x)}{\mu(m_{s,t_N})}\right\vert \leq   \alpha \frac{ m_{s,t_N}(x)}{\mu(m_{s,t_N})}\prod_{i\leq N} (1-c_id_i)
\ee
and using twice this bound proves
\eqref{viveineg} by triangular inequality.

Finally, \eqref{interm} also gives, for $\tau=t$,
$$\frac{ m_{s,t_N}(x)}{\mu(m_{s,t_N})} \leq    \frac{ m_{s,t}(x)}{\mu(m_{s,t})}\frac{ 1}{1- \alpha \prod_{i\leq N} (1-c_id_i)}. $$
Then \eqref{viveineg} implies \eqref{onemore}. 
\end{proof}

Using the previous results, we now prove the existence of harmonic functions and Theorem~\ref{th:main}.

\begin{proof}[Proof of Theorem~\ref{th:main}.]  
We  fix  $s\geq 0$, $\nu \in \P(\X_s)$  and $\beta>0$. We begin by proving that there exists a function $h_s$ positive and bounded  such that for any $x\in \X_s$ and any $t\geq s$,
\be
\label{maj}
 \bigg\vert \frac{m_{s,t}(x)}{\nu(m_{s,t})} - h_s(x)     
 \bigg\vert  \leq  2\alpha \e^{-C_{\alpha,\beta, \nu}(s,t)}\min \left\{\beta,  \frac{m_{s,t}(x)}{\nu(m_{s,t})} \frac{ 1}{\left(1-\alpha \e^{-C_{\alpha,\beta, \nu}(s,t)}\right)_+}\right\}.
 \ee
First, optimizing Inequality \eqref{viveineg} over all the admissible coupling constants yields
\begin{equation}
\label{eq:cauchyseq}
\left\vert \frac{m_{s,\tau}(x)}{\nu(m_{s,\tau})}-\frac{ m_{s,t}(x)}{\nu (m_{s,t})}\right\vert \leq 2\beta \alpha \e^{-C_{\alpha,\beta, \nu}(s,t)}
\end{equation}
by recalling  Definition \ref{eq:coupling-capacity} and that  \eqref{as:A4I}   guarantees $m_{s,t_N}(x)/\nu(m_{s,t_N})\leq \beta$. \\
Using that $C_{\alpha,\beta, \nu}(s,t)\rightarrow\infty$ as $t\rightarrow \infty$, Cauchy criterion  ensures that  the following limit exists
\be
\label{hsdef}
h_s(x)=\lim_{\tau\rightarrow \infty}   \frac{m_{s,\tau}(x)}{\nu (m_{s,\tau})}.
\ee
Moreover, letting $\tau\rightarrow \infty$ in \eqref{eq:cauchyseq} shows that
$$\left\vert \frac{ m_{s,t}(x)}{\nu(m_{s,t})}-h_s(x)\right\vert \leq 2\beta \alpha \e^{-C_{\alpha,\beta, \nu}(s,t)}.$$
Optimizing  now similarly over coupling constants in \eqref{onemore} and letting $\tau\rightarrow \infty$  yields
$$\left\vert\frac{ m_{s,t}(x)}{\nu(m_{s,t})}-h_s(x)\right\vert \leq   \frac{ m_{s,t}(x)}{\nu(m_{s,t})} \frac{ 2\alpha \e^{-C_{\alpha,\beta, \nu}(s,t)}}{1-\alpha \e^{-C_{\alpha,\beta, \nu}(s,t)}}, $$
for any $t$ such that $\alpha \exp(-C_{\alpha,\beta, \nu}(s,t))<1$.
Combining these two bounds proves $(\ref{maj})$.

Integrating  $(\ref{maj})$ over some $\mu\in \M_+(\X_s)$, we get
\be
\label{controlmoyenne}
 \big\vert \mu (m_{s,t}) - \mu(h_s) \nu(m_{s,t})     \big\vert  \leq 2 \alpha \min \left\{\beta \mu(\X) \nu(m_{s,t}),  \frac{\mu (m_{s,t})}{(1-\alpha \e^{-C_{\alpha,\beta, \nu}(s,t)})_+}\right\}  \e^{-C_{\alpha,\beta, \nu}(s,t)}.
 \ee
Moreover Lemma~\ref{denis} $(iii)$ yields (after optimization over coupling constants) for  non-zero $\mu$,
\be
\label{optim}
\biggl\|\frac{\mu M_{s,t}}{\mu (m_{s,t})}-\frac{\nu M_{s,t}}{\nu (m_{s,t})}\biggr\|_{\mathrm{TV}}\leq 2 \e^{-C_{\alpha,\beta, \nu}(s,t)} 
\ee 
and 
combining the two previous inequalities  
gives
\begin{align*}
&\biggl\| \mu  M_{s,t}  -  \mu(h_s) \nu(m_{s,t})  \frac{ \nu M_{s,t} }{ \nu(m_{s,t})} \biggr\|_{\mathrm{TV}} \\
&\qquad\leq \biggl\| \mu  M_{s,t}  -  \mu(m_{s,t})\frac{\nu M_{s,t}}{\nu(m_{s,t})}\biggr\|_{\mathrm{TV}}+
\big|\mu(m_{s,t}) - \mu(h_s)\nu(m_{s,t})\big| \left\|\frac{\nu M_{s,t} }{\nu(m_{s,t})}\right\|_{\mathrm{TV}}\\
&\qquad \leq  2 \left(\mu(m_{s,t})+ \alpha \min \left\{\beta \mu(\X) \nu(m_{s,t}),  \frac{\mu (m_{s,t})}{(1-\alpha \e^{-C_{\alpha,\beta, \nu}(s,t)})_+}\right\}\right)\e^{-C_{\alpha,\beta, \nu}(s,t)}.
\end{align*}

Using again Inequality \eqref{optim}, with $\mu= \gamma M_{s_0,s}$, we obtain
\begin{align}
\bigg\| &\mu  M_{s,t}  -  \mu(h_s)\nu(m_{s,t}) \frac{ \gamma M_{s_0,t} }{ \gamma(m_{s_0,t})}  \bigg\|_{\mathrm{TV}}  \label{exp} \\
& \leq 2 \left(\mu(m_{s,t})+  \mu(h_s) \nu(m_{s,t})+ \alpha \min \left\{\beta \mu(\X) \nu(m_{s,t}),  \frac{\mu (m_{s,t})}{(1-\alpha \e^{-C_{\alpha,\beta, \nu}(s,t)})_+}\right\}\right)\e^{-C_{\alpha,\beta, \nu}(s,t)}. \nonumber
\end{align}
To conclude it remains to control $\mu(m_{s,t})$ and $\mu(h_s)$. First, we notice that $h_s$ is bounded by $\beta$ using \eqref{as:A4I} and \eqref{hsdef}.
Using again (\ref{as:A4I}), we have
$$
\mu(m_{s,t})\leq \beta \mu(\X)  \nu(m_{s,t})
$$
and the first bound of \eqref{exp}  yields
\begin{equation*}
\biggl\| \mu  M_{s,t}  -  \mu(h_s) \nu(m_{s,t})  \frac{ \gamma M_{s_0,t} }{ \gamma(m_{s_0,t})} \biggr\|_{\mathrm{TV}} 
 \leq  2(2 +\alpha)\beta \mu(\X)  \nu(m_{s,t}) \e^{-C_{\alpha,\beta, \nu}(s,t)}.
\end{equation*}
Moreover,  if $C_{\alpha,\beta, \nu}(s,t)> \log(3\alpha)$, using the second part of (\ref{controlmoyenne}) and the fact that $\vert a-b\vert \leq \eta |b|$ and $\eta\in[0,1)$ imply that $|b|\leq |a|/(1-\eta)$
ensures that   
$$
\mu(m_{s,t}) \leq  \frac{1- \alpha \e^{-C_{\alpha,\beta, \nu}(s,t)}}{1- 3 \alpha \e^{-C_{\alpha,\beta, \nu}(s,t)}} \mu(h_s)  \nu(m_{s,t}),
$$
so that the second part of~\eqref{exp} becomes
\begin{align*}
\biggl\| \mu  M_{s,t}  -  \mu(h_s) \nu(m_{s,t})  \frac{ \gamma M_{s_0,t} }{ \gamma(m_{s_0,t})} \biggr\|_{\mathrm{TV}} 
 &\leq  2 \frac{2+\alpha- 4\alpha \e^{-C_{\alpha,\beta, \nu}(s,t)}}{1- 3 \alpha \e^{-C_{\alpha,\beta, \nu}(s,t)}} \mu(h_s)  \nu(m_{s,t}) \e^{-C_{\alpha,\beta, \nu}(s,t)}.
\end{align*}
This proves the estimate stated in Theorem \ref{th:main} when $C_{\alpha,\beta, \nu}(s,t)\geq \log(4\alpha)$.
Finally, this estimate applied to   $\mu=\delta_x$ ensures that
\[ |m_{s,t}(x)  -  h_s(x) \nu(m_{s,t})| \leq  8(2+\alpha)  h_s(x)\,\nu(m_{s,t}) \, \e^{-C_{\alpha,\beta, \nu}(s,t)}\]
and then
\[ \big(1 + 8(2+\alpha) \e^{-C_{\alpha,\beta, \nu}(s,t)}\big)h_s(x)\geq \frac{m_{s,t}(x)}{\nu(m_{s,t})}>0,\]
so that $h_s>0$.
The fact that $\nu(h_s)=1$ follows directly from~\eqref{hsdef} and dominated convergence theorem, while uniqueness of $h_s$ is derived letting $t$ go to infinity.
\end{proof}

\section{Applications}
\label{applis}

In the present section, we develop different applications of Theorem \ref{th:main}.
We aim at  illustrating the main result and show how to check the required assumptions.
Yet we also obtain new estimates and mention that the models  can be made more complex.

We first  consider the heat equation with growth and reflecting boundary on the compact set $[0,1]$ and time space inhomogeneity. 
The coupling capacity is then expressed in terms of the function describing the diffusion coefficient. \\ 
Then we prove general statements when the semigroup is homogeneous, asymptotically homogeneous, and periodic. 
The results  are illustrated by asymptotic estimates for the renewal equation.

\subsection{A growth-diffusion equation with reflecting boundary and varying environment}
\label{sect:diff}

In this section $\X_t=\X= [0,1]$ for every $t\geq 0$. We consider a population of particles which reproduce and move following a diffusion varying in time. The evolution of its density is prescribed by the following PDE
\begin{equation}
\label{eq:heat}
\left \{
\begin{array}{r @{\ =\ } l}
    \partial_t u_{s,t}(x) & {1 \over 2} \sigma_t \Delta u_{s,t}(x) + r(x) u_{s,t}(x), \qquad 0 < x < 1, 
    \vspace{2mm}\\
    \partial_x u_{s,t}(0) & \partial_x u_{s,t}(1)= 0,
    \vspace{2mm}\\
   u_{s,s}(x)& \phi(x),
\end{array}
\right.
\end{equation}
for some $\phi \in L^1([0,1]).$ As usual, we do not stress the dependence on $\phi$ of $u.$
This equation is the nonautonomous Heat Equation with growth under Neumann boundary conditions.
More precisely, particles diffuse with coefficients $(\sigma_t)_{t\geq 0}$ on the space $[0,1]$.
The growth rate $r(x)$ is the difference between birth and death rate at position $x\in [0,1].$

In this example,  $\sigma$ is time-dependent but not space-dependent and conversely for $r$. Our   coupling methods  provide a relevant approach for estimating the speed of convergence in this varying environment case. This analysis could be easily generalized for both time-dependant and space-dependant parameters but would provide tedious computations, so it is left for future works.

Let us work with another representation of the solution of \eqref{eq:heat}. Let $(X^x_{s,t})_{t\geq s}$ be a reflected Brownian motion on $[0,1]$ starting from $x$ at time $s$, with diffusion coefficient $\sigma_t$ at time $t$, see \eqref{defX} below for a construction.  We define the positive semigroup $M$ by
\begin{equation}
\label{eq:FK}
M_{s,t} f = \mathbb{E}_{x}\left[ f(X^x_{s,t}) e^{\int_s^t r(X^x_{s,u}) du} \right]
\end{equation}
for every bounded Borel function $f$ on $[0,1],$ and $t\geq s \geq 0.$
Then $\mu M_{s,t}$ is defined by setting for all $f\in \B_b([0,1])$
\[(\mu M_{s,t})(f)=\mu(M_{s,t}f).\]
 Feynman-Kac formula \cite[Chapter VII Proposition (3.10) p.358]{RY} states the duality relation of this semigroup with the solution $u$ of \eqref{eq:heat}:  
$$\int_0^1 \phi(x) M_{s,t} f(x) dx =\int_0^1 u_{s,t}(x)f(x)dx$$
for every bounded measurable function $f.$
This property allows to see the mapping $t\to\mu M_{s,t}$ as the unique solution to Equation~\eqref{eq:heat} when the initial density $\phi$ is replaced by a measure $\mu.$

\subsubsection{Statements}\

We assume that $t\mapsto \sigma_t$ is a non-negative and c\`adl\`ag function, $r$ is continuous and
$$
-\infty<\underline{r}:=\inf_{x \in [0,1]}  r(x); \qquad  \overline r:=\sup_{x\in [0,1]} r(x) <+ \infty.
$$
Introduce the  function $g$ with value in $\mathbb{R} \cup \{+\infty \}$ to measure the coupling capacity in function of the parameters
$$g :(s,t) \mapsto  (\overline{r}-\underline{r})(t-s)-\log\left(\left(1-4/\sigma_{s,t} \right)_+ \right),$$
where
$$\sigma_{s,t}:=\sqrt{2\pi \int_{s}^{t} \sigma_u^2du}.$$
These functions allow to control  the coupling capacity in this model by considering
$$\mathfrak C_{\tau,\rho}(s,t)=\sup_{  \mathcal T_{\tau, \rho}(s,t)} \left\{ -\sum_{i=1}^{N} \log\Big(1- \exp\big(-\big(g(t_{i-1},t_{i})+g(t_{i},t_{i+1})\big)\big)\Big)\right\},$$
where $\mathcal T_{\tau, \rho}(s,t)$ is the set of subdivisions $(t_i)_{i=0}^{N+1}$ such that  $N\geq  1$, $s=t_0 \leq \cdots \leq t_{N+1}\leq t$ and
$$  t_1- t_{0}\leq \rho, \  t_{N}- t_{N-1}\leq \tau, \  t_{N+1}- t_{N}\leq \tau \ \  \text{ and } \ \  \text{for } i\in \{0,N-1,N\},    \ \sigma_{t_{i},t_{i+1}} \geq 5.$$
 Indeed, $\mathfrak C_{\tau,\rho}(s,t)$ is  a lower bound of \eqref{eq:coupling-capacity}. The constant $5$ may be  improved  and replaced for instance by $4 + \varepsilon$, but we restrict ourselves here to this value  which allows to get a simple expression  of the coupling constants $\alpha$ and $\beta$, namely 
 $\alpha=\gamma_{\tau}^2,  \beta=\gamma_{\rho}\gamma_{\tau}$
 where
 $$ \gamma_{s}=5\exp( (\overline{r}-\underline{r})s) \in [5,\infty).$$
  The first time interval of size $\rho$ is involved in the control of the mass and the expression of $\beta$. A general quantitative bound
 can now be given as follows, writing $\lambda$  the Lebesgue measure on $[0,1]$.

\begin{theo} \label{th:reflected-mb}
Let $s\geq 0$ and $\tau>0$. Assume that $\mathfrak C_{\tau,\rho}(s,t)\rightarrow \infty$ as $t\rightarrow \infty$.
Then there exists a function $h_s : [0,1] \rightarrow (0, \gamma_{\rho}\gamma_{\tau}]$  and  probabilities  $(\pi_{t})_{t\geq0}$ such that
$$ \biggl\| \mu  M_{s,t}  -  \mu(h_s) \lambda(m_{s,t}) \pi_{t}  \biggr\|_{\mathrm{TV}} \leq 8(2+\gamma_{\tau}^2) \vert \mu \vert (h_s) \lambda(m_{s,t}) \e^{-\mathfrak C_{\tau, \rho}(s,t)},$$
for any $\mu\in \mathcal M([0,1])$ and $t$ such that $\mathfrak C_{\tau,\rho}(s,t) \geq 2\log(2)+2\log(\gamma_{\tau})$. 
\end{theo}

The proof of Theorem \ref{th:reflected-mb} is postponed to Subsection \ref{sect:proof-diff}. Let us now illustrate this result by constructing a relevant lower bound of $\mathfrak C_{\tau,\rho}(s,t)$. Let $s\geq 0, \tau>0$ and  set
$$t_1(s,\tau)=\inf\{ u \geq s :  \sigma_{s,u} \geq 10.\}$$
and the sequence $(t_k(s,\tau))_k$ defined by induction :
$$t_{k+1}(s,\tau)=\inf\{ u\geq t_k(s,\tau)+\tau : \sigma_{u,u+\tau} \geq 10\} \qquad (k\geq 1),$$
using again the convention $\inf\varnothing=+ \infty$.
As $(s,t)\mapsto \sigma_{s,t}$ is continuous,  there exists $t_k'(s,\tau)$ such that
$$\sigma_{ t_k(s,\tau), t_k'(s,\tau)}\geq 5 \ \text{ and } \ \sigma_{t_k'(s,\tau), t_k(s,\tau)+\tau}\geq 5.$$
Using then the time subdivision $s,t_1(s,\tau),..., t_k(s,\tau), t_k'(s,\tau), t_k(s,\tau)+\tau,...$  $(k\geq 2)$, we get an upper  bound for $g(t_k'(s,\tau)-t_k(s,\tau))$ and $g(t_k(s,\tau)+\tau-t_k'(s,\tau)) $and
$$
\mathfrak C_{\tau,t_1(s,\tau)-s}(s,t)\geq -\log(1-1/\gamma_{\tau}^2) \max\{ N : t_{N+1}(s,\tau) \leq t-\tau\}.
$$
We  derive then immediately a  speed of convergence from Theorem \ref{th:reflected-mb}.

\begin{coro} 
Let $s\geq 0$ and $\tau>0$. Assume that $t_k(s,\tau)<\infty$ for any $k\geq 1$.\\
Then there exists a positive bounded function $h_s$ and probabilities $(\pi_t)_{t\geq 0}$ such that 
$$\liminf_{t\to \infty} -\frac{1}{t} \log  \biggl\| \frac{\mu  M_{s,t}}{\lambda(m_{s,t})}   -  \mu(h_s)  \pi_t \biggr\|_{\mathrm{TV}} \geq  -\log(1-1/\gamma_{\tau}^2).
\liminf_{k\rightarrow \infty} \frac{k}{t_k(s,\tau)},$$
uniformly over $\mu \in \P([0,1])$. Moreover $h_s$ is bounded by $\gamma_{\tau} \gamma_{t_1(s,\tau)-s}$. 
\end{coro}
As soon as $\limsup_{i} t_i(s,\tau)/i<\infty$, we obtain an exponential speed. 
Note also that superexponential or subexponential  speed could be obtained by alternative  constructions of time sequences.

As an application for exponential convergence in random environment, let us consider a Feller c\`adl\`ag Markov process  $(t,w)\in [0, \infty)\times \Omega \rightarrow \sigma_t(w)\in [0,\infty)$  on a probability space $(\Omega, \mathcal F,  \mathbb P)$. We
assume that the process $(\sigma_t)_{t\geq 0}$  is Harris positive reccurent with stationary probability $\pi\neq \delta_0$. For each $w\in \Omega$, we write $M_{s,t}=M_{s,t}(w)$ the semigroup defined by \eqref{eq:FK} and associated to the diffusion coefficient $(\sigma_t(w))_{t\geq 0}$. 
By a Birkhoff Theorem (see \cite[Theorem~8.1]{MTII} and \cite[Theorem~3.2]{MTII}), we obtain that $\max\{ N : t_{N+1}(s,\tau) \leq t-\tau\}$ grows linearly as $t\rightarrow\infty$ with a deterministic speed. It yields the following quenched estimate: there exists $v>0$ such that
$$\mathbb P\left(\liminf_{t\to \infty} -\frac{1}{t} \log \left(\sup_{\mu \in  \mathcal P([0,1]) }  \biggl\| \frac{\mu  M_{s,t}}{\lambda(m_{s,t})}   -  \mu(h_s)  \pi_t  \biggr\|_{\mathrm{TV}}\right) \geq v \right)=1.$$
and the convergence is  uniform and (at least) exponential. Let us observe that the diffusion may be zero for arbitrarily large time intervals.

\subsubsection{Proof of Theorem \ref{th:reflected-mb}}
\label{sect:proof-diff}
We begin by the construction and some useful estimates for  process  $(X^x_{s,t})_{s,t}$, which is  associated to a nonhomogeneous
diffusion $[0,1]$ with Neumann boundary condition. Let $(B_t)_{t\geq 0}$ be a classical Brownian motion on $\mathbb{R}$, and set $(W^x_{s,t})_{t\geq s \geq 0}$ defined for all $x\in [0,1]$ and $t\geq s\geq 0$ by
\be
\label{defX}
W^x_{s,t} = x + \int_s^t \sigma_u dB_u.
\ee
The random variable $W_{s,t}^x$ is distributed according to a Gaussian law $\mathcal{N}\left(x,\sigma_{s,t}\right).$ The reflected process $(X^x_{s,t})$ can now be defined by
$$
\forall t\geq s \geq 0, \quad X^x_{t,s} =\sum_{n \in \mathbb{Z}} (W^x_{s,t}-2n) \mathbf{1}_{W^x_{s,t} \in [2n,2n+1]} + (2n-W^x_{s,t}) \mathbf{1}_{W^x_{s,t} \in [2n-1,2n]}.
$$

\begin{lem}[Bounds on the density for the diffusion]
\label{lem:diffbound} For any $t> s \geq 0$, there exists $c_{s,t}\geq 1$ such that for any Borel set $A$ of $[0,1]$,
$$ 
\left(c_{s,t}-\frac{4}{\sigma_{s,t}} \right)_{+} \lambda(A) \leq  \mathbb{P}(X^x_{s,t} \in A) \leq c_{s,t} \lambda(A),
$$
with the convention $1/0=\infty$.
\end{lem}
\begin{proof}
We define
$$\phi^{x}_{s,t}(y)=\frac{1}{\sqrt{2\pi \sigma_{s,t}}} \exp\left(-\frac{(y-x)^2}{2\sigma_{s,t}}\right)$$
the density of  $W_{s,t}^x$ .
Using that
\begin{equation*}
\mathbb{P}_x(X_{s,t} \in A)=\sum_{n\in \Z} \mathbb P(W_t^x\in (A+2n)\cap [2n,2n+1])+\mathbb P(W_t^x\in (2n-A)\cap [2n-1,2n]),
\end{equation*}
we obtain
\begin{equation}
\label{eq:MB-cd}
d_{s,t}^x(A) \lambda(A) \leq  \mathbb{P}_x(X_{s,t} \in A) \leq c_{s,t}^x(A) \lambda(A), 
\end{equation}
with
\Bea
c^x_{s,t}(A)&=& \sum_{n\in \Z} \left(\sup_{ A +2n} \phi^x_{s,t}+\sup_{ 2n-A} \phi^x_{s,t}\right),\\
d^x_{s,t}(A)&=& \sum_{n\in \Z} \left(\inf_{ A +2n} \phi^x_{s,t}+\inf_{ 2n-A} \phi^x_{s,t}\right).
\Eea
Using  \eqref{eq:MB-cd}, these constants verify
$$c_{s,t}:= \sup_{x\in [0,1]} c_{s,t}^x([0,1]) \geq c^x_{s,t}(A)\geq 1 \geq d^x_{s,t}(A) \geq \inf_{x\in [0,1]} d^x_{s,t}([0,1])=:d_{s,t}. $$
It remains to prove an upperbound of  the difference $c_{s,t}-d_{s,t}$ and conclude. Decomposing the sum following $n=0, n\geq 1$ and $n\leq -1$ shows on a first hand that
$$c_{s,t}=\sup_{x\in [0,1]} \sum_{n\in \Z} \sup_{[n,n+1]}\phi^x_{s,t}=
\sup_{x\in [0,1]}\left( \phi_{s,t}^0(0)+\sum_{n\in \Z} \phi^0_{s,t}(n-x) \right),
$$
and on the other hand
\begin{align*} 
d_{s,t}&=  \inf_{x\in [0,1]}\sum_{n\in \Z} \inf_{[n,n+1]}\phi^x_{s,t} 
=
\inf_{x\in [0,1]}\left(\phi^0_{s,t}(\max(x,1-x)) + \sum_{n\in \Z \setminus\{0,1\}}  \phi^0_{s,t}(n-x)\right).
\end{align*}
Combining 
\begin{align*}
c_{s,t}
&\leq \phi^0_{s,t}(0) + \sum_{n\in \Z}  \sup_{x \in [0,1]} \phi^0_{s,t}(n-x)
= 2 \phi^0_{s,t}(0) + \sum_{n \in \mathbb{Z}} \phi^0_{s,t}(n),
\end{align*}
and
\begin{align*}
d_{s,t} 
&\geq \inf_{x \in [0,1]} \phi^0_{s,t}(\max(x,1-x)) + \sum_{n\in \Z \setminus \{0,1\} } \inf_{x \in [0,1]} \phi^0_{s,t}(n-x) \\
&= \phi^0_{s,t}(1)  + \sum_{ n \in \mathbb{Z} \setminus \{-1,0,1 \} } \phi^0_{s,t}(n) =  \sum_{ n \in \mathbb{Z} \setminus \{-1,0 \} } \phi^0_{s,t}(n),
\end{align*}
we obtain $c_{s,t} -d_{s,t} \leq 4 \phi^0_{s,t}(0).$ This ends the proof.
\end{proof}

We give now the coupling constants and introduce
$$\overline{\sigma}_{s,t}=\left(1-\frac{4}{\sigma_{s,t }}\right) \mathbf{1}_{\sigma_{s,t }>4}.$$

\begin{lem}[Coupling constants]
\label{lem:reflected-mb}
Let $s= t_0 \leq \cdots \leq t_N\leq t_{N+1}\leq t $. \\
(i) Assumptions \eqref{as:A1I} and \eqref{as:A2I} are satisfied with the constants
 $$
 c_i=\overline{\sigma}_{t_{i-1},t_i}\e^{-(\overline{r}-\underline{r})(t_{i}-t_{i-1})}, \quad i=1,\dots,N+1.
 $$
and
$$d_i=c_{i+1}= \overline{\sigma}_{t_i,t_{i+1}}\e^{-(\overline{r}-\underline{r})(t_{i+1}-t_i)}, \quad i=1,...,N,$$
 and $\nu_i=\lambda$ for $i=1, \ldots, N+1$. \\
(ii) If  $\sigma_{t_{0}, t_{1}}  > 4$, 
$\sigma_{t_{N-1}, t_{N}}  > 4,\, \sigma_{t_{N}, t_{N+1}} > 4,$ 
then \eqref{as:A3I} and \eqref{as:A4I} hold with 
$$
\alpha= \frac{ \e^{(\overline{r}-\underline{r})(t_{N+1}-t_{N-1})}}{\overline{\sigma}_{t_{N-1},t_N}\overline{\sigma}_{t_{N},t_{N+1}}}, \qquad \beta =  \frac{\e^{(\overline{r}-\underline{r})(t_{1}-t_0)}}{ \overline{\sigma}_{t_0,t_{1}}},$$
and $\nu=\lambda$.
\end{lem}
\begin{proof}
Lemma \ref{lem:diffbound} and Equation \eqref{eq:FK} ensure that for any probability measure $\mu$, \ $u > v \geq 0$
$$
 d_{u,v} \e^{\underline{r}(v-u)} \lambda \leq \mu M_{u,v} \leq c_{u,v} \e^{\overline{r}(v-u)} \lambda,
$$
where $d_{u,v}:=\left( c_{u,v}-4/\sigma_{u,v}\right)_+$. Then
\begin{equation} 
\label{eq:brownien-encadrement12}
 \overline{\sigma}_{u,v}c_{u,v} \e^{\underline{r}(v-u)} \lambda \leq \mu M_{u,v} \leq c_{u,v} \e^{\overline{r}(v-u)} \lambda,
\end{equation}
recalling that $c_{u,v}\geq 1$.
In particular
\begin{equation}
\label{eq:encad-m}
 \overline{\sigma}_{u,v} c_{u,v} \e^{\underline{r}(v-u)}  \leq \mu(m_{u,v}) \leq c_{u,v} \e^{\overline{r}(v-u)}.
\end{equation}
Combining the two last bounds, we find 
$$
\mu M_{u,v} \geq   \overline{\sigma}_{u,v} c_{u,v} \e^{\underline{r}(v-u)} \lambda \geq  \overline{\sigma}_{u,v} \e^{-(\overline{r}-\underline{r})(v-u)} \, \mu(m_{u,v}) \lambda 
$$
and obtain the expected value of $c_i$, with $u=t_{i-1}, v=t_i$. \\
Moreover for all $w\in (u,v]$ such that $d_{u,w}>0$,  we also have by integration of  \eqref{eq:brownien-encadrement12}
$$
 \overline{\sigma}_{u,w}c_{u,w} \e^{\underline{r}(w-u)} \lambda (m_{w,v}) \leq \mu(m_{u,w}) \leq   c_{u,w} \e^{\overline{r}(w-u)} \lambda (m_{w,v}).
$$
 Taking respectively  $\mu=\delta_x$ and $\mu= \lambda$ for the second (resp. first) inequality,
\begin{equation}
\label{eq:tprime}
 \overline{\sigma}_{u,w}  \e^{-(\overline{r}-\underline{r})(w-u)} m_{u,v}(x)\leq  \lambda (m_{u,v})
\end{equation}
and we get the value of $d_i$, with $u=t_i, w=t_{i+1}, v=\tau$. 

Let us turn to the proof of $(ii)$. Using again \eqref{eq:tprime} with now $u=t_N, w=t_{N+1}, v=\tau$ yields
$$ m_{t_N,T}(x)\leq \frac{1}{\overline{\sigma}_{t_N,t_{N+1}}} \e^{-(\overline{r}-\underline{r})(t_{N+1}-t_N)}  \lambda (m_{t_N,\tau}). $$
Recalling the expression of $c_N$ from $(i)$ provides the expected value of $\alpha$. 
Finally, using  \eqref{eq:tprime} with $u=s, w=t_{1}, v=t_N$ yields  $\beta$.
\end{proof}

Remark that the previous proof could be achieved, without probabilistic notation. Indeed, bounding $r$, one can build sub and super solutions, which, up to renormalization, satisfy Equation \eqref{eq:heat} with $r=0$. It is then enough to use Lemma \ref{lem:diffbound} which is only based on the explicit solutions of \eqref{eq:heat} with constant $r$.
\begin{proof}[Proof of Theorem \ref{th:reflected-mb} ] Consider a sequence 
 $(t_i)_{i=0}^{N+1}$ such that  $N\geq  1$, $s=t_0 \leq \cdots \leq t_{N+1}\leq t$ and
$$  t_0- t_{1}\leq \rho, \  t_{N}- t_{N-1}\leq \tau, \  t_{N+1}- t_{N}\leq \tau \ \  \text{ and } \ \  \text{for } i\in \{0,N-1,N\},    \ \sigma_{t_{i+1},t_{i}} \geq 5.$$
By Lemma \ref{lem:reflected-mb}, the following constants
$$c_i=\e^{-g(t_{i-1},t_i)},   \qquad d_i=\e^{-g(t_{i},t_{i+1})}$$
are $(\alpha,\beta,\lambda)$ admissible coupling constants for $M$ on $[s,t]$, with $\alpha=\gamma_{\tau}^2$ and  $\beta=\gamma_{\rho}.$
Then, optimizing over  admissible coupling constants yields
 $$  C_{\alpha,\beta,\lambda}(s,t)\geq \mathfrak C_{\tau,\rho}(s,t),
$$
and applying  Theorem \ref{th:main}  ends the proof.
\end{proof}

\subsection{Homogeneous semigroups and the renewal equation}
\label{sect:homogene}
First, we specify the general result in the homogeneous setting, which simplifies the assumptions. Second, we develop an application to the renewal equation.

\subsubsection{Existence of eigenelements and speed of convergence}
In this subsection we consider a semigroup $(M_{s,t})_{0\leq s\leq t}$ which is homogeneous,
meaning that the sets $\X_t=\X$ are  time-independent and that there exists a semigroup $(M_t)_{t\geq 0}$ set on $\X$ such that $M_{s,t}=M_{t-s}$ for all $t\geq s \geq 0$.
In this case Assumptions \eqref{as:A1I}, \eqref{as:A2I}, \eqref{as:A3I} and \eqref{as:A4I} can be simplified as follows.

First,   there exist  constants $c>0$, $r>0$ and a probability measure $\nu$ such that for any $x\in \X$,
\begin{equation}\label{as:H1-h}\tag{H1}
\delta_x M_{r}\geq c\, m_{r}(x)\, \nu.
\end{equation}
Second, there exists a constant $d>0$ such that for any $t\geq 0$,
\begin{equation}\label{as:H2-h}\tag{H2}
\nu( m_{t}) \geq d\, \| m_{t}\|_{\infty},
\end{equation}
where we recall the notation $m_t=M_t\1.$ These assumptions have been obtained in \cite{CV16}  for the study of process conditionned on non-absorption.  
Let us note that \cite{CV16} also prove that they are necessary conditions for uniform exponential convergence in total variation distance.
Additionally the main result can be strengthened, provided an additional assumption:
\begin{equation}\label{as:H3-h}\tag{H3}
\text{The function}\ t\mapsto\|m_t\|_\infty\ \text{is locally bounded on}\ \R_+.
\end{equation}
Notice that this last assumption is satisfied by  classical semigroups  appearing in applications.

\begin{theo}
\label{th:homogene}
Under Hypotheses~\eqref{as:H1-h},~\eqref{as:H2-h} and~\eqref{as:H3-h}, there exists a unique triplet $(\gamma, h, \lambda) \in \P(\X)\times \B_b(\X)\times\R$ such that $\gamma(h)=1$ and for all $t\geq0$,
\begin{equation}
\label{eq:vp-hom}
\gamma M_t = \e^{\lambda t} \gamma\qquad\text{and}\qquad \ M_t h = \e^{\lambda t} h.
\end{equation}
Additionally $h$ is bounded and positive on $\X$ and there exists $C>0$ such that for all $t\geq0$ and  $\mu \in \mathcal{M}(\X)$,
$$
\left\Vert \e^{-\lambda t}\mu M_t - \mu (h) \gamma \right\Vert_{\mathrm{TV}} \leq C\, \|\mu\|_{\mathrm{TV}} \left(1 -cd \right)^{t/r}.
$$
\end{theo}

Theorem~\ref{th:homogene} strictly extends the main result of~\cite{CV14}.
For instance their theorems do not apply for the semigroup of Section~\ref{sect:hom-renewal} below.
This semigroup cannot be written as a the semigroup of a killed Markov process, even up to exponential normalisation.

\begin{Rq}
By differentiating \eqref{eq:vp-hom}, the triplet $(\gamma,h,\lambda)$ is a triplet of eigenelements for the infinitesimal generator of $(M_t)_{t\geq0},$
that is $\gamma\mathcal A=\lambda\gamma$ and $\mathcal A h=\lambda h$ where the (unbounded) operator $\mathcal A$ is defined by
$\mathcal A=\lim_{t\to0}\frac1t(M_t-I).$ \\
\end{Rq}

\begin{Rq}
With Lemma~\ref{denis} $(iii)$, we also recover (as expected) the statement from \cite{CV16} :
$$
\left\Vert \frac{\mu M_t}{\mu(m_t)} - \gamma \right\Vert_{\mathrm{TV}} \leq 2 \left(1 -cd \right)^{t/r}.
$$
\end{Rq}

\begin{proof}
Assume ~\eqref{as:H1-h} and~\eqref{as:H2-h} and  consider $(t_i)_{0\leq i \leq N}$ defined by $t_i=ir$.
Then  Assumptions \eqref{as:A1I} - \eqref{as:A4I} hold with, for every $0\leq i \leq N$,
\[c_i=c, \quad d_i=d, \quad \nu_i=\nu, \quad \alpha =\frac{1}{cd}, \quad \beta = \frac{1}{d}.\]
This implies
$$C_{\alpha,\beta,\nu}(0,t)\geq \lfloor (t-r)/r\rfloor \log(1-cd).$$
By  Theorem~\ref{th:main} applied to $M$ on $[0,t]$,  there exist $C>0$ and $h_0 : \X\rightarrow [0,\infty) $  such that for any  $\mu,\widetilde\mu \in \mathcal M(\X)$ 
\be
\label{casgen}
\biggl\| \mu  M_{t}  -  \mu(h_0) \nu(m_{t}) \frac{\widetilde\mu M_t}{\widetilde\mu(m_t)}  \biggr\|_{\mathrm{TV}} \leq C\, \nu(m_t)\|\mu\|_{\mathrm{TV}} \left(1 -cd \right)^{t/r},
\ee
Moreover, taking $\mu=\delta_x,$
$$h_0=\lim_{T\rightarrow \infty} \frac{m_T}{\nu(m_T)}$$
for the supremum norm and $h_0$ is positive, bounded by $1/d,$ and $\nu(h_0)=1.$
By the semigroup property we have $m_{t+s}(x)= \delta_x M_t m_s$ from which we deduce that
\[\frac{m_{t+s}(x)}{\nu(m_{t+s})}\frac{\nu M_t m_s}{\nu(m_s)}= \delta_x M_t\Big(\frac{m_s}{\nu(m_s)}\Big).\]
Letting $s\to \infty$,  we get 
$$(\nu M_t h_0) h_0 = M_t  h_0.$$
This means that $h_0$ is an eigenvector of $M_t$ associated to the eigenvalue $\nu M_t  h_0.$
Since the semigroup property yields
\[\nu M_{t+s}  h_0=\nu M_{t} M_sh_0=  (\nu M_t  h_0) .( \nu M_s  h_0)\]
and $t \mapsto \nu M_t  h_0$ is locally bounded (by assumption~\eqref{as:H3-h}, because $0\leq\nu M_th_0\leq\|m_t\|_\infty/d$),
we deduce the existence of $\lambda\in \mathbb{R}$ such that $$\nu M_t h_0 = \e^{\lambda t}\nu ( h_0)= \e^{\lambda t}.$$

Let us now show the existence of a left eigenvector $\gamma.$
Applying~\eqref{casgen} to $\mu=\nu$ and $\widetilde\mu=\nu M_s$ we get that
\[\biggl\| \frac{\nu  M_{t}}{\nu(m_t)}  -  \frac{\nu M_{t+s}}{\nu(m_{t+s})}  \biggr\|_{\mathrm{TV}}\leq C \left(1 -cd \right)^{t/r}.\]
This ensures that the family $\big(\frac{\nu  M_{t}}{\nu(m_t)}\big)_{t\geq0}$ satisfies the Cauchy property
and we deduce the existence of a probability measure $\gamma$ such that
\[\biggl\| \frac{\nu  M_{t}}{\nu(m_t)}  -  \gamma \biggr\|_{\mathrm{TV}}\leq C \left(1 -cd \right)^{t/r}.\]
Then we use the semigroup property to write that for all $s,t\geq0$ we have
\[\frac{\nu M_s}{\nu (m_s)}M_t=\nu M_t\bigg(\frac{m_s}{\nu(m_s)}\bigg) \frac{\nu M_{s+t}}{\nu (m_{s+t})}\]
and letting $s$ tend to infinity we find
$$\gamma M_t=(\nu M_t  h_0)\,\gamma=\e^{\lambda t}\gamma.$$

Now we set $h=h_0/\gamma(h_0),$ so that $\gamma(h)=1.$
Applying~\eqref{casgen} to $\mu=\widetilde\mu=\gamma$ and dividing by $\nu(m_t)$ yields
\[\bigg|\frac{\e^{\lambda t}}{\nu(m_t)}-\gamma(h_0)\bigg|\leq C(1-cd)^{t/r}.\]
Finally,  using~\eqref{casgen} with $\widetilde\mu=\gamma,$ we write for $\mu\in\M(\X)$ and $t\geq0$
\begin{align*}
\|\e^{-\lambda t}\mu M_t-\mu(h)\gamma\|_{\mathrm{TV}}
&\leq\nu(m_t)\e^{-\lambda t}\bigg(\bigg\|\frac{\mu M_t}{\nu(m_t)}-\mu(h_0)\gamma\bigg\|_{\mathrm{TV}}
+|\mu(h)|\bigg|\gamma(h_0)-\frac{e^{\lambda t}}{\nu(m_t)}\bigg|\bigg)\\
&\leq C'\nu(m_t)\e^{-\lambda t}\|\mu\|_{\mathrm{TV}}(1-cd)^{t/r}.
\end{align*}
The conclusion follows from the fact that the function $t\mapsto\nu(m_t)\e^{-\lambda t}$ is bounded.
Indeed it is locally bounded due to~\eqref{as:H3-h}, and it converges to $1/\gamma(h_0)$ when $t\to+\infty.$
\end{proof}

\subsubsection{Example: the renewal equation}
\label{sect:hom-renewal}
We consider an age-structured population of proliferating cells which divide at age $a\geq0$ according to a division rate $b(a),$
giving birth to two daughter cells with age zero.
The evolution of the age distribution density $u_t$ is given by the so-called renewal PDE
\begin{equation}\label{eq:renewal}\left\{\begin{array}{ll}
\partial_t u_t(a)+\partial_a u_t(a)+b(a)u_t(a)=0, &\qquad t,a>0,
\vspace{2mm}\\
\dis u_t(0)=2\int_0^\infty b(a)u_t(a)\,da, & \qquad t>0.
\end{array}\right.\end{equation}
This model has been introduced by Sharpe and Lotka~\cite{SharpeLotka}
in a more general context, namely with a ``birth rate'' not necessarily equal to twice the ``death rate''.
Since then, it has become a very popular model in population dynamics (see for instance~\cite{Arino,Iannelli,MetzDiekmann,Pert07,Thieme,Webb}).

The state space here is $\X=\R_+=[0,\infty).$
Following~\cite{G17} we associate to Equation~\eqref{eq:renewal} the homogeneous semigroup $(M_t)_{t\geq0}$ defined as follows.
For any $f\in \B_b(\R_+)$, we define the family $(M_tf)_{t\geq0}\subset \B_b(\R_+)$ as the unique solution to the  equation
\begin{equation}\label{eq:renewal:Duhamel}
M_tf(a)=f(a+t)\e^{-\int_0^tb(a+\tau)d\tau}+2\int_0^t\e^{-\int_0^\tau b(a+\tau')d\tau'}b(a+\tau)M_{t-\tau}f(0)\,d\tau.
\end{equation}
The proof of the existence and uniqueness of a solution to \eqref{eq:renewal:Duhamel} is postponed in Appendix~\ref{appendix:renewal}, Lemma~\ref{prop:renewal:Duhamel}.
We also refer to Appendix~\ref{appendix:renewal} for the rigorous definition of $\mu M_t,$ which provides the unique measure solution to Equation~\eqref{eq:renewal} with initial distribution $\mu.$
In particular if $\mu$ has a density $u_0$ with respect to the Lebesgue measure,
we get that $u_t=\mu M_t$ is the unique $L^1$ solution to Equation~\eqref{eq:renewal} with initial distribution $u_0.$
Appendix~\ref{appendix:renewal} also contains a verification of Assumption~\ref{as:Mst} for the semigroup $(M_t)_{t\geq0}.$

Now we can use Theorem~\ref{th:homogene} to obtain the long time asymptotic behavior of the solutions to Equation~\eqref{eq:renewal}.
\begin{theo}
\label{th:renewal:convergence}
Assume that $b$ is a non-negative locally bounded function on $\R_+,$
and suppose the existence of $a_0>0$, $p>0$, $l\in(p/2,p],$ and $\underline b>0$ for which
\[\forall k\in\N,\forall a\in[a_0+kp,a_0+kp+l],\qquad b(a)\geq\underline b.\]
Then there exists a unique triplet of eigenelements $(\gamma, h, \lambda) \in \P(\R_+)\times \B_b(\R_+)\times \R_+$ verifying $\gamma(h)=1$ and
\[\forall t\geq 0, \quad \gamma M_t = \e^{\lambda t} \gamma, \ M_t h = \e^{\lambda t}h.\]
Moreover there exist $C>0$ and an explicit $\rho>0,$ given by~\eqref{defdegeu}, such that for all $\mu \in \mathcal{M}(\R_+)$ and all $t\geq0$
\[\left\Vert \e^{-\lambda t}\mu M_t - \mu (h) \gamma \right\Vert_{\mathrm{TV}} \leq C \|\mu\|_{\mathrm{TV}}\, \e^{-\rho t}.\]
\end{theo}

The convergence of the solutions to a time-independent asymptotic profile multiplied by an exponential function of time, sometimes referred to as \emph{asynchronous exponential growth}, was first conjectured for the renewal equation by Sharpe and Lotka~\cite{SharpeLotka} and was then proved by many authors using various methods, see for instance~\cite{Feller,Greiner,GwiazdaPerthame,GwiazdaWiedemann,KhaladiArino,Pert07,Song,Webb84}.
Moreover it is known that the so-called Malthus parameter $\lambda$ is characterized as the unique real number which satisfies the characteristic equation
\[1=2\int_0^\infty b(a)\e^{-\int_0^a(\lambda+b(a'))da'}da,\]
the asymptotic probability measure has an explicit density with respect to the Lebesgue measure
\[\gamma(da)=\kappa\,\e^{-\int_0^a(\lambda+b(a'))da'}da\]
where $\kappa$ is a normalization constant which ensures that $\gamma\in\P(\R_+),$
and the harmonic function $h$ is explicitly given by
\[h(a)=2h(0)\int_a^\infty b(a')\e^{-\int_a^{a'}(\lambda+b(a''))da''}da',\]
where $h(0)$ is chosen so that $\gamma(h)=1.$

The existing results about asynchronous exponential growth for the renewal equation hold for birth and death rates which are not necessarily related by a multiplicative factor, as we assume here.
But in these previous works the birth rate is assumed to be bounded or integrable, a condition which is not required in our situation.
In Equation~\eqref{eq:renewal} if the division rate is unbounded, then the unboundedness of the birth rate $2b$ is ``compensated'' by the unboundedness of the death rate $b.$
Thus our result is new in the sense that the assumptions on the division rate are very general, but also because it provides an explicit spectral gap in terms of the division rate and a convergence which is valid for measure solutions.

The assumptions on the division rate $b$ include some functions which are not bounded neither from above nor from below by a positive constant when $a$ tends to infinity.
The only assumption is that, outside a compact interval, $b$ is larger than a crenel function with period $p$ and a crenel width $l.$
The condition $l>p/2$ is only a technical assumption which simplifies the computations.
It can be removed to the price of a larger number of iterations of the Duhamel formula in the proof.

Before proving Theorem~\ref{th:renewal:convergence}, we define   the probability distribution of age of  division
\[\Phi(a):=b(a)\e^{-\int_0^a b(a')da'}\]
and we give a useful property of $m_t=M_t\1.$

\begin{lem}\label{lm:renewal:mt}
For any $a\geq0$ the function $t\mapsto m_t(a)$ is non-decreasing.
\end{lem}
\begin{proof}
First we check that $m_t(0)\geq1$ for all $t\geq0.$
By definition $t\mapsto m_t(0)$ is the unique fixed point of
\[\Gamma g(t)=\e^{-\int_0^tb(\tau)d\tau}+2\int_0^t \Phi(\tau)g(t-\tau)\,d\tau\]
and if $g\geq1$ then for all $t\geq0$
\[\Gamma g(t)\geq\e^{-\int_0^tb(\tau)d\tau}+2\int_0^t \Phi(\tau)\,d\tau=2-\e^{-\int_0^tb(\tau)d\tau}\geq1.\]
So the fixed point necessarily satisfies $m_t(0)\geq1,$ since $\Gamma$ is a contraction for small times (see Lemma~\ref{prop:renewal:Duhamel}).

In a second step we prove that $t\mapsto m_t(0)$ is non-decreasing.
Let $\epsilon>0.$
For all $t\geq0$ we have by definition of $m_t(0)$
\begin{align*}
m_{t+\epsilon}(0)-m_t(0)
&=-\int_t^{t+\epsilon}\Phi(\tau)\,d\tau+2\int_t^{t+\epsilon} \Phi(\tau)m_{t+\epsilon-\tau}(0)\,d\tau\\
&\hskip40mm +2\int_0^t\Phi(\tau)(m_{t+\epsilon-\tau}(0)-m_{t-\tau}(0))\,d\tau\\
&=\int_t^{t+\epsilon} \Phi(\tau)(2m_{t+\epsilon-\tau}(0)-1)\,d\tau+2\int_0^t\Phi(\tau)(m_{t+\epsilon-\tau}(0)-m_{t-\tau}(0))\,d\tau.
\end{align*}
So $t\mapsto m_{t+\epsilon}(0)-m_t(0)$ is the unique fixed point of
\[\Gamma g(t)= f_0(t)\e^{-\int_0^tb(\tau)d\tau}+2\int_0^t \Phi(\tau)g(t-\tau)\,d\tau\]
with $f_0(t)=\int_t^{t+\epsilon} \e^{-\int_t^\tau b(\tau')d\tau'}b(\tau)(2m_{t+\epsilon-\tau}(0)-1)\,d\tau\geq0.$
We deduce from the positivity property in Lemma~\ref{prop:renewal:Duhamel} that $m_{t+\epsilon}(0)-m_t(0)\geq0$ for all $t\geq0.$

The last step consists in extending the result of the second step to $t\mapsto m_t(a)$ for any $a\geq0.$
Let $a\geq0$ and $\epsilon>0.$
For all $t\geq0$ we have
\begin{align*}
m_{t+\epsilon}(a)-m_t(a)=&\int_t^{t+\epsilon} \e^{-\int_t^\tau b(a+\tau')d\tau'}b(a+\tau)(2m_{t+\epsilon-\tau}(0)-1)\,d\tau\\
&+2\int_0^t\e^{-\int_t^\tau b(a+\tau')d\tau'}b(a+\tau)(m_{t+\epsilon-\tau}(0)-m_{t-\tau}(0))\,d\tau\geq0.
\end{align*}
\end{proof}

\begin{coro}\label{cor:renewal:mt}
For all $t,a\geq0$ we have $m_t(a)\leq2\,m_t(0).$
\end{coro}

\begin{proof}
Starting from the Duhamel formula~\eqref{eq:renewal:Duhamel} and using Lemma~\ref{lm:renewal:mt} we have
\begin{align*}
m_t(a)&=\e^{-\int_0^tb(a+\tau)\,d\tau}+2\int_0^t\e^{-\int_0^sb(a+\tau)\,d\tau}b(a+s)m_{t-s}(0)\,ds\\
&\leq \e^{-\int_0^tb(a+\tau)\,d\tau}+2m_t(0)\int_0^t\e^{-\int_0^sb(a+\tau)\,d\tau}b(a+s)\,ds\\
&= \e^{-\int_0^tb(a+\tau)\,d\tau}+2m_t(0)\biggl[1-\e^{-\int_0^tb(a+\tau)\,d\tau}\biggr]\\
&= 2m_t(0)+\e^{-\int_0^tb(a+\tau)\,d\tau}(1-2m_t(0))\leq2m_t(0).
\end{align*}
\end{proof}

We are now ready to prove Theorem~\ref{th:renewal:convergence}.

\begin{proof}[Proof of Theorem~\ref{th:renewal:convergence}]
We prove that Assumptions~\eqref{as:H1-h},~\eqref{as:H2-h}, and~\eqref{as:H3-h} are satisfied by the renewal semigroup and then apply Theorem \ref{th:homogene}.

We start with~\eqref{as:H1-h}.
For $\alpha>0$, we define the probability measure $\nu$ by
\[\forall f\in C_0(\R_+),\qquad \nu(f):=\frac{\int_0^\alpha M_sf(0)\,ds}{\int_0^\alpha m_s(0)\,ds}.\]
We want to prove that for $\alpha$ small enough (to be determined later), there exists a time $t_0>0$ and $c>0$ such that for all $f\geq0$ and  $a\geq0$,
\be
\label{lamin}
M_{t_0}f(a)\geq c\,\nu(f)m_{t_0}(a).
\ee
Iterating the Duhamel formula~\eqref{eq:renewal:Duhamel} we have for all $f\geq0$ and all $t,a\geq 0$,
\begin{align*}
M_tf(a)&=f(a+t)\e^{-\int_0^t b(a+\tau)\,d\tau}+2\int_0^t\e^{-\int_0^\tau b(a+\tau')\,d\tau'}b(a+\tau)f(t-\tau)\e^{-\int_0^{t-\tau}b(\tau')\,d\tau'}\,d\tau\\
&\hspace{8mm}+4\int_0^t\e^{-\int_0^\tau b(a+\tau')\,d\tau'}b(a+\tau)\int_0^{t-\tau}\Phi(\tau')M_{t-\tau-\tau'}f(0)\,d\tau'd\tau\\
&\geq4\int_0^t\e^{-\int_0^\tau b(a+\tau')\,d\tau'}b(a+\tau)\int_0^{t-\tau}\Phi(t-\tau-s)M_sf(0)\,ds\,d\tau.
\end{align*}
Let $t_0>0,$ $\alpha\in[0,t_0],$ and $0\leq t_1\leq t_2\leq t_0-\alpha.$ We have
\[M_{t_0}f(a)\geq4\int_{t_1}^{t_2}\e^{-\int_0^\tau b(a+\tau')\,d\tau'}b(a+\tau)\int_0^\alpha\Phi(t_0-\tau-s)M_sf(0)\,ds\,d\tau.\]
This inequality means that for bounding $M_{t_0}f$ from below we only keep the individuals which:
do not divide between times $0$ and $t_1;$
divide a first time between $t_1$ and $t_2;$
do not divide between $t_2$ and $t_0-\alpha;$
divide a second time between $t_0-\alpha$ and $t_0.$

Let us check that we can choose $t_0,t_1,t_2$ and $\alpha$ such that $b$ and then $\Phi$ have a lower bound on $[t_0-t_2-\alpha,t_0-t_1]$ and  then give a lower bound for \[\int_{t_1}^{t_2}\e^{-\int_0^\tau b(a+\tau')\,d\tau'}b(a+\tau)\,d\tau=\e^{-\int_a^{a+t_1}b(\tau)\,d\tau}\big(1-\e^{-\int_{a+t_1}^{a+t_2}b(\tau)\,d\tau}\big)\]
and obtain \eqref{lamin}. For that purpose, 
we define $n=\big\lfloor a_0/p\big\rfloor+1$ the smallest integer such that $np> a_0.$
Let $\alpha\in(0,2l-p)$ and $t_0=a_0+np+l.$
The choice of $t_1$ and $t_2$  depends on whether $a<a_0$ or $a\geq a_0.$\\
For $a<a_0$ we choose $t_1=np$ and $t_2=np+l-\alpha.$
We have $b\geq \underline b$ on $[t-t_2-\alpha,t-t_1]=[a_0,a_0+l],$
so that $\Phi\geq \underline b\,\e^{-\int_0^{a_0+l}b(\tau)d\tau}$ on $[t-t_2-\alpha,t-t_1],$
and
\begin{align*}
\int_{t_1}^{t_2}\e^{-\int_0^sb(a+\tau)\,d\tau}b(a+s)\,ds
&\geq \e^{-\int_0^{a_0+np} b(\tau)d\tau}\big(1-\e^{-\underline b(2l-p-\alpha)}\big)>0.
\end{align*}
For $a\geq a_0$ we choose $t_1=0$ and $t_2=l-\alpha.$
We have $\Phi\geq \underline b\,\e^{-\int_0^{A+np+l}b(\tau)d\tau}>0$ on $[t-t_2-\alpha,t-t_1]=[a_0+np,a_0+np+l],$ and
\[\int_0^{t_2}\e^{-\int_0^s b(a+\tau)\,d\tau}b(a+s)\,ds=1-\e^{-\int_a^{a+l-\alpha}b(\tau)\,d\tau}\geq 1-\e^{-\underline b(2l-p-\alpha)}>0.\]
As a consequence, \eqref{lamin} is satisfied with
\[c=4\frac{\int_0^\alpha m_s(0)\,ds}{\|m_{t_0}\|_\infty}\,\underline b\,\big(1-\e^{-\underline b(2l-p-\alpha)}\big)\,\e^{-2\int_0^{a_0+np+l}b(\tau)d\tau},\]
which gives  Assumption~\eqref{as:H1-h}.
\medskip

Now we turn to~\eqref{as:H2-h}.
Since we know from Corollary~\ref{cor:renewal:mt} that $m_t(0)\geq\frac12m_t(a)$ for all $t,a\geq0,$
it suffices to find $d>0$ such that for all $t\geq0$,
\[\nu(m_t)\geq2\,d\,m_t(0).\]
Lemma~\ref{lm:renewal:mt} ensures that for all $t\geq0$,
\[\nu(m_t)=\frac{1}{\int_0^\alpha m_s(0)ds}\int_0^\alpha m_{t+s}(0)ds
\geq\frac{\alpha}{\int_0^\alpha m_s(0)ds} m_t(0)\]
and the constant $d=\frac{\alpha}{2\int_0^\alpha m_s(0)ds}$ suits.

\medskip

It remains to check~\eqref{as:H3-h}.
In that view,  we define
\[\mathfrak b(a):=\sup_{[0,a]}b\]
and we write for $t\geq s\geq0$
\[m_s(0)=\e^{-\int_0^sb(\tau)\,d\tau}+2\int_0^s\e^{-\int_0^{s-\tau}b(\tau')\,d\tau'}b(s-\tau)m_{\tau}(0)\,d\tau\leq 1+2\,\mathfrak b(t)\int_0^sm_{\tau}(0)\,d\tau.\]
Applying the Gr\"onwall's lemma we get $m_s(0)\leq \e^{2 \mathfrak b(t) s}$ for all $s\in[0,t],$ so $m_t(0)\leq \e^{2 \mathfrak b(t) t},$
and using Corollary~\ref{cor:renewal:mt} we obtain 
\[\|m_t\|_\infty\leq2\,\e^{2 \mathfrak b(t)t}.\]

Finally we can apply Theorem~\ref{th:homogene} which ensures the exponential convergence with the rate
$$
\frac{-\log(1-cd)}{t_0}=\frac{-\log\left(1-\frac{2\alpha}{\|m_{t_0}\|_\infty}\underline b\,\big(1-\e^{-\underline b(2l-p-\alpha)}\big)\,\e^{-2\int_0^{a_0+np+l}b(\tau)d\tau}\right)}{a_0+np+l}\\
\geq \rho$$
where 
\be
\label{defdegeu}
\rho= \frac{-\log\left(1-\alpha\underline b\,\big(1-\e^{-\underline b(2l-p-\alpha)}\big)\,\e^{-2\int_0^{2a_0+p+l}b(\tau)d\tau-2(2a_0+p+l) \mathfrak b(2a_0+p+l)}\right)}{2a_0+p+l}
\ee
and the result follows by choosing $\alpha=l- p/2.$
\end{proof}

\subsection{Asymptotically homogeneous semigroups and  increasing maximal age}
\label{sect:quasih}

In this section, we present a general theorem for semigroups which become homogeneous when time tends to infinity. We then apply this theorem to an age structured population where 
the state space has a maximal age which increases with time.

\subsubsection{Convergence of the  profile and evolution of the mass}

We consider the situation of a semigroup which becomes homogeneous when time goes to $\infty$.
For the sake of simplicity and in view of our application, we restrict ourselves to the case when the state space is increasing :$$\forall s< t, \quad \X_s\subset \X_t, \qquad \X=\bigcup_{s\geq 0} \X_s.$$
We say that a semigroup $(M_{s,t})_{0\leq s\leq t}$  is asymptotically homogeneous if there exists a homogeneous semigroup $(N_t)_{t\geq0}$ defined on $\X$ and satisfying Assumption \ref{as:Mst}, such that for all $s\geq0$
\begin{equation}\label{as:H0-ah}\tag{H'0}
\lim_{t\to \infty} \sup_{x \in \X_t}  \left\Vert \delta_x M_{t,t+s}-\delta_x N_s\right\Vert_{\mathrm{TV}} = 0
\end{equation}
In our framework, the assumptions  \eqref{as:A1I}-\eqref{as:A4I} rewrite as follows. There exist $s_0,\ro>0$, $c,d>0$ and some probability  measure $\nu$ on $\X_{s_0}$ such that
for any $t\geq s_0$ and $x\in \X_t$,
\begin{equation}\label{as:H1-ah}\tag{H'1}
\delta_x M_{t,t+\ro}\geq c \, m_{t,t+\ro}(x) \, \nu,
\end{equation}
and  for any  $\tau\geq 0$,
\begin{equation}\label{as:H2-ah}\tag{H'2}
d\,m_{t,t+\tau}(x) \leq \nu(m_{t,t+\tau}).
\end{equation}
As for the homogeneous case, writing $n_{t}(x)=\delta_x N_{t}\mathbf{1}$, for $x \in \X$ and $t\geq 0$, we need that
\begin{equation}\label{as:H3-ah}\tag{H'3}
 t\mapsto\|n_{t}\|_\infty\ \text{is locally bounded on }  \R_+.
\end{equation}

\begin{theo}\label{th:asymptotic-hom}
Let $s\geq 0$. Under Assumptions \eqref{as:H0-ah}, \eqref{as:H1-ah}, \eqref{as:H2-ah} and \eqref{as:H3-ah}, there exists a probability measure $\gamma$ on $\X$ and a positive bounded function $h_s$ on $\X_s$ such that 
$$\lim_{t\to \infty} \sup_{\mu \in\P(\X_s)} \biggl\| \frac{\mu  M_{s,t}}{ \nu(m_{s,t})}-  \mu(h_s)    \gamma  \biggr\|_{\mathrm{TV}}  = 0.$$
\end{theo}
Notice that the TV norm above is the TV norm on the state space $\X,$ the measure $\mu M_{s,t}\in\M(\X_t)$ being extended by zero on $\X\setminus\X_t.$

Notice also that in Theorem~\ref{th:asymptotic-hom} we do not provide a speed of convergence.
It could be achieved by taking into account the speed of convergence of 
$M$ to $N$.
\begin{proof}
Following \eqref{Pstt}, we set  
$$\delta_x Q_{s,t}^{(t)}=\frac{\delta_x N_{t-s}}{n_{t-s}(x)}.$$
for any $x \in \X$, $t\geq s\geq 0$.
Fix $x_0\in \X$.  First, using  $\eqref{as:H0-ah}$,  for any $u\geq 0$, we obtain
 $$
 \lim_{t \to \infty} m_{t,t+u}(x_0)= n_{u}(x_0)>0
 $$
and then using again $\eqref{as:H0-ah}$,
\begin{equation}
\label{quasiaux}
\lim_{t\to \infty} \left\Vert \delta_{x_0} P^{(t+u)}_{t,t+u}-\delta_{x_0} Q^{(t+u)}_{t,t+u} \right\Vert_{\mathrm{TV}} = 0.
\end{equation}

Now we note that \eqref{as:H1-ah} and \eqref{as:H2-ah} ensure that \eqref{as:A1I}-\eqref{as:A4I} are satisfied for any regular subdivision of $[s,t]$ with step $\ro$ ({\it i.e.}  $t_i-t_{i-1}=\ro$) and $t_0\geq s_0,$ for the constants
$$c_i=c, \qquad d_i=d, \qquad \nu_i=\nu, \qquad \alpha=1/(cd), \qquad \beta=1/d.$$
We apply then Theorem \ref{th:main} to $M$ with $\gamma=\delta_{x_0}$, which ensures that, for every $s\geq 0,$ 
\be
\label{thmcor}
\lim_{t\to \infty} \sup_{\mu\in \mathcal P(\X_s)}   \biggl\| \frac{\mu  M_{s,t}}{ \nu(m_{s,t})}  -  \mu(h_s)   \delta_{x_0}P^{(t)}_{0,t}   \biggr\|_{\mathrm{TV}}=0,
\ee
while Lemma \ref{denis} $(ii)$ guarantees
\be
\label{ergaux}
\lim_{u\to \infty}
\sup_{t\geq 0, x,y \in \X_t}
\|\delta_xP_{t,t+u}^{(t+u)}-\delta_yP_{t,t+u}^{(t+u)}\|_{\mathrm{TV}}   = 0.
\ee
Using $\eqref{as:H0-ah}$ and letting $t\rightarrow \infty$ in  \eqref{as:H1-ah}, \eqref{as:H2-ah} and \eqref{as:H3-ah}, we obtain that the semigroup $N$ satifies \eqref{as:H1-h}, \eqref{as:H2-h} and  \eqref{as:H3-h}. 
Then by Theorem \ref{th:homogene}, there exists a probability measure $\gamma$ such that
$$
 \lim_{u \to \infty} \sup_{ x \in \X, t\geq 0}\|\delta_xQ_{t,t+u}^{(t+u)}-\gamma \|_{\mathrm{TV}}   = 0.
$$
Hence, \eqref{quasiaux} becomes
\[
\lim_{u\rightarrow \infty}  \limsup_{t\rightarrow\infty} \|\delta_{x_0}P_{t,t+u}^{(t+u)}-\gamma \|_{\mathrm{TV}}   = 0.
\]
Using now   $\delta_{x_0}P_{0,t+u}^{(t+u)}= (\delta_{x_0}P_{0,t}^{(t+u)})P_{t,t+u}^{(t+u)}$ and \eqref{ergaux}, we get
$$\lim_{u\rightarrow\infty} \sup_{t\geq 0} \|\delta_{x_0} P_{0,t+u}^{(t+u)}-\delta_{x_0}P_{t,t+u}^{(t+u)} \|_{\mathrm{TV}}   = 0.$$
Combing the two last bounds yields
\Bea
&&\limsup_{t\to\infty}  \|  \delta_{x_0}P^{(t)}_{0,t}-\gamma \|_{\mathrm{TV}} = \lim_{u\to\infty} \sup_{t\geq 0} \|  \delta_{x_0}P^{(t+u)}_{0,t+u}-\gamma \|_{\mathrm{TV}} \\
&&\qquad  \leq \lim_{u\rightarrow\infty} \sup_{t\geq 0} \|\delta_{x_0} P_{0,t+u}^{(t+u)}-\delta_{x_0}P_{t,t+u}^{(t+u)} \|_{\mathrm{TV}}  +
 \lim_{u\rightarrow \infty}  \limsup_{t\rightarrow\infty} \|\delta_{x_0}P_{t,t+u}^{(t+u)}-\gamma \|_{\mathrm{TV}}  = 0.
\Eea
Plugging this in \eqref{thmcor} concludes the proof. 
\end{proof}

\subsubsection{Example: the renewal equation with maximal age}
We consider the renewal equation on a domain bounded by a maximal age which increases along time.
When an individual reaches the maximal age, he dies.
We denote by $a_t$ the maximal age, which grows from $a_0$ to $a_\infty$ when $t\to+\infty.$
To avoid pathological situations we assume that $t\mapsto t-a_t$ is strictly increasing.
Inside the domain $[0,a_t)$ the individuals reproduce with a birth rate $b(a)$ bounded from below by $\underline b>0$ and from above by $\overline b.$
The partial differential equation which prescribes the evolution of the density $u_{s,t}(a)$ of individuals with age $a$ at time $t$ (starting from a distribution $u_s(a)$ at time $s\in[0,t)$) writes
\be\label{eq:age_max}\left\{\begin{array}{ll}
\partial_t u_{s,t}(a)+\partial_au_{s,t}(a)=0,& t>s,\ 0< a<a_t,
\vspace{2mm}\\
\dis u_{s,t}(0)=\int_0^{a_t}\!b(a)u_{s,t}(a)\,da,\qquad\qquad& t>s,
\vspace{2mm}\\
u_{s,s}(a)=u_s(a),& 0\leq a< a_s.
\end{array}\right.\ee
The motivation of such a model comes from \cite{BUS} which studies  a related branching model for the bus paradox problem.

The details of the construction of the associated semigroup are postponed to Appendix~\ref{appendix:maxage}.
Here  we only give the main steps.
For $t\geq s\geq0$ we define the operator $M_{s,t}:\B_b([0,a_t))\to \B_b([0,a_s))$ as follows:
for any $f\in \B_b([0,a_t))$ the family $(M_{s,t}f)_{t\geq s\geq0}$ is the unique solution to the equation
\be\label{eq:Duhamel_age_max}
M_{s,t}f(a)=f(a+t-s)+\int_s^t b_\tau(a+\tau-s)M_{\tau,t}f(0)\,d\tau,
\ee
where we have denoted $b_t(a)=b(a)\1_{[0,a_t)}(a)$ and $f$ has been extended by 0 beyond $a_t.$
Then for $\mu\in\M(\X_s)$ the measure $\mu M_{s,t}$ is defined on $[0,a_t)$ in such a way that $(\mu M_{s,t})(f)=\mu(M_{s,t}f)$ for all $f\in \B_b([0,a_t)).$

Using Theorem~\ref{th:asymptotic-hom} we prove the following ergodic result for the semigroup $(M_{s,t})_{t\geq s\geq0}$. It provides the long time asymptotic behavior of the measure solutions $(\mu M_{s,t})_{t\geq s}$ to \eqref{eq:age_max}.

\medskip

\begin{theo}\label{th:agemax:convergence}
Let $s\geq0.$
There exist $\gamma\in\P([0,a_\infty)),$ $\nu\in\P([0,a_s)),$ and a positive $h_s\in \B_b([0,a_s))$ such that for all $\mu\in\M([0,a_s))$,
$$\lim_{t\to \infty} \sup_{\mu \in\P(\X_s)} \biggl\| \frac{\mu  M_{s,t}}{ \nu(m_{s,t})}-  \mu(h_s)    \gamma  \biggr\|_{\mathrm{TV}}  = 0.$$
\end{theo}

Before proving this theorem we start with a useful lemma. 
\begin{lem}\label{lm:agemax:Gronwall}
For any $s\leq t$ and any $f\in C_0([0,a_t))$, we have
\[\|M_{s,t}f\|_\infty\leq\e^{\overline b(t-s)}\|f\|_\infty.\]
\end{lem} 

\begin{proof}
By definition of $M_{s,t}f(0)$ we have
\[|M_{s,t}f(0)|=\left|f(t-s)+\int_s^t b_\tau(\tau-s)M_{\tau,t}f(0)\,d\tau\right|\leq \|f\|_\infty+\overline b\int_s^t|M_{\tau,t}f(0)|\,d\tau\]
and the Grönwall's lemma gives $|M_{s,t}f(0)|\leq\|f\|_\infty\e^{\overline b(t-s)}.$
Then for $a\geq0$ we write
\[|M_{s,t}f(a)|\leq\|f\|_\infty+\overline b\|f\|_\infty\int_s^t\e^{\overline b(\tau-s)}d\tau=\e^{\overline b(t-s)}\|f\|_\infty.\]
\end{proof}

\begin{proof}[Proof of Theorem~\ref{th:agemax:convergence}]
We start by verifying~\eqref{as:H0-ah} with $\X= [0,a_\infty).$
The homogeneous semigroup $N$ on $\M(\X) \times \B_b(\X)$ is defined by the following Duhamel formula
\[N_{t}f(a)=f(a+t)+\int_0^t b_\infty(a+\tau)N_{t-\tau}f(0)\,d\tau.\] 
Existence, uniqueness and Assumption \ref{as:Mst} can be proved in the same way as for the homogeneous renewal equation (see Appendix~\ref{appendix:renewal}).

Fix $t>0,$ $a\in\X_t=[0,a_t),$ and $f\in C_0\big([0,a_t)\big)$ such that $\|f\|_\infty\leq1.$
For any $r\geq a_\infty$ we have
\[M_{t,t+r}f(a)=\int_t^{t+r}b_\tau(a+\tau-t)M_{\tau,t+r}f(0)\,d\tau\]
and
\[N_rf(a)=\int_t^{t+r}b_\infty(a+\tau-t)N_{t+r-\tau}f(0)\,d\tau.\]
We start by comparing $M_{t,t+r}f(0)$ and $N_rf(0).$
We have, using Lemma~\ref{lm:agemax:Gronwall},
\begin{align*}
&|N_rf(0)-M_{t,t+r}f(0)|\\
& \quad \leq
\int_t^{t+r}\Big|b_\infty(\tau-t)-b_\tau(\tau-t)\Big|\,|N_{t+r-\tau}f(0)|\,d\tau\\
&\hspace{20mm}+\int_t^{t+r}b_\tau(\tau-t)|N_{t+r-\tau}f(0)-M_{\tau,t+r}f(0)|\,d\tau\\
&\quad \leq\int_t^{t+r}\overline b\,\1_{a_\tau\leq\tau-t\leq a_\infty}\e^{\overline b(t+r-\tau)}d\tau
+\overline b\int_t^{t+r}|N_{t+r-\tau}f(0)-M_{\tau,t+r}f(0)|\,d\tau\\
& \quad \leq\overline b\, \e^{\overline b r}(a_\infty-a_t)+\overline b\int_t^{t+r}|N_{t+r-\tau}f(0)-M_{\tau,t+r}f(0)|\,d\tau,
\end{align*}
which gives by Gr\"onwall's lemma
\[|N_rf(0)-M_{t,t+r}f(0)|\leq \overline b\, \e^{2\overline b r}(a_\infty-a_t).\]
Now we come back to $a\in[0,a_t)$ and we have similarly
\begin{align*}
|N_rf(a)-M_{t,t+r}f(a)|&\leq
\overline b \e^{\overline b r}(a_\infty-a_t)+\overline b\int_t^{t+r}|N_{t+r-\tau}f(0)-M_{\tau,t+r}f(0)|\,d\tau\\
&\leq\overline b\, \e^{\overline b r}(a_\infty-a_t)+{\overline b}^2r\e^{2\overline b r}(a_\infty-a_t)\\
&\leq\max\Big(\overline b \e^{\overline b r},\big(\overline b \e^{\overline b r}\big)^2r\Big)(a_\infty-a_t).
\end{align*}
We deduce that for all $r\geq a_\infty$,
\[\sup_{0\leq a<a_t}\|\delta_aM_{t,t+r}-\delta_aN_r\|_{\mathrm{TV}}\leq\max\Big(\overline b \e^{\overline b r},\big(\overline b \e^{\overline b r}\big)^2r\Big)(a_\infty-a_t)
\xrightarrow[t\to+\infty]{}0.\]

Now we turn to~\eqref{as:H1-ah}.
Iterating the Duhamel formula~\eqref{eq:Duhamel_age_max}, we get for $f\geq0$,
\begin{align*}
M_{s,t}f(a)&=f(a+t-s)+\int_s^t b_\tau(a+\tau-s)M_{\tau,t}f(0)\,d\tau\\
&=f(a+t-s)+\int_s^t b_\tau(a+\tau-s) f(t-\tau)\,d\tau\\
&\hspace{4mm}+\int_s^t b_\tau(a+\tau-s)\int_\tau^t b_{\tau'}(\tau'-\tau)f(t-\tau')\,d\tau'd\tau +(\geq0).
\end{align*}
Thus
\begin{align*}
M_{s,t}f(a) &\geq\int_s^t b_\tau(a+\tau-s)
\int_\tau^t b_{\tau'}(\tau'-\tau)f(t-\tau')\,d\tau'd\tau.
\end{align*}
We consider $s$ large enough so that $\Delta:=a_\infty-a_s\leq a_s/2.$
We set $\alpha:=\frac\Delta2$ and take $t-s=2\Delta.$
For $s\leq\tau\leq \tau'\leq t$ we have $\tau'-\tau\leq t-s=2\Delta\leq a_s\leq a_\tau'$ and so $b_{\tau'}(\tau'-\tau)=b(\tau'-\tau)\geq\underline b.$
We deduce that
\[M_{s,t}f(a)\geq\underline b\int_s^t b_\tau(a+\tau-s)\int_\tau^tf(t-\tau')\,d\tau'd\tau.\]
We consider separately $a\leq a_s-\alpha$ and $a\geq a_s-\alpha.$
For $a\leq a_s-\alpha$ et $\tau\in[s,s+\alpha]$ we have $a+\tau-s\leq a_s\leq a_\tau $ and so $b_\tau(a+\tau-s)=b(a+\tau-s)\geq\underline b.$
Thus we can write
\begin{align*}
M_{s,t}f(a)
&\geq\underline b\int_s^t b_\tau(a+\tau-s)
\int_\tau^t f(t-\tau')\,d\tau'd\tau\\
&\geq\underline b^2\int_s^{s+\alpha}\int_{t-\alpha}^t f(t-\tau')\,d\tau'd\tau
\geq\alpha \underline b^2\int_0^\alpha f(r)\,dr=(\alpha\underline b)^2\nu(f),
\end{align*}
where $\nu(f)=\frac1\alpha\int_0^\alpha f(r)\,dr.$
Now for $a>a_s-\alpha$ we have that if $\tau\geq t-\alpha$ then $a+\tau-s\geq a_s+\Delta=a_\infty$ and so $b_\tau(a+\tau-s)=0.$
We deduce that
\begin{align*}
M_{s,t}f(a)
&\geq\underline b\int_s^t b_\tau(a+\tau-s)
\int_\tau^t f(t-\tau')\,d\tau'd\tau\\
&\geq\underline b\int_s^t b_\tau(a+\tau-s)
\int_{t-\alpha}^t f(t-\tau')\,d\tau'd\tau=\alpha\underline b\bigg(\int_s^t b_\tau(a+\tau-s)
d\tau\bigg)\nu(f).
\end{align*}
To conclude that~\eqref{as:H1-ah} is satisfied it remains to compare these lower bounds with $m_{s,t}(a).$
We start from the Duhamel formula
\[m_{s,t}(a)=\1_{a+t-s< a_t}+\int_s^t b_\tau(a+\tau-s)m_{\tau,t}(0)\,d\tau\]
and use Lemma~\ref{lm:agemax:Gronwall} which ensures that $m_{\tau,t}(0)\leq \e^{\overline b(t-\tau)}.$
If $a\leq a_s-\alpha$, we write that
\[m_{s,t}(a)\leq1+\overline b\int_s^t\e^{\overline b(t-\tau)}d\tau=\e^{2\overline b\Delta}.\]
For $a>a_s-\alpha$ we use that $a+t-s\geq a_s+\Delta=a_\infty\geq a_t$ to write
\[m_{s,t}(a)\leq\int_s^tb_\tau(a+\tau-s)m_{\tau,t}(0)\,d\tau\leq\e^{\overline b\Delta}\int_s^tb_\tau(a+\tau-s)\,d\tau.\]
Finally we get
\[M_{s,t}f(a)\geq\min\big(\alpha\underline b\, \e^{-\overline b\Delta},(\alpha\underline b\, \e^{-\overline b\Delta})^2\big) \,m_{s,t}(a)\,\nu(f).\]

\medskip

For~\eqref{as:H2-ah} we start by comparing $m_{s,t}(a)$ to $m_{s,t}(0).$
We have
\begin{align*}
m_{s,t}(a)&=\1_{a+t-s< a_t}+\int_s^t b_\tau(a+\tau-s)m_{\tau,t}(0)\,d\tau\\
&\leq\1_{t-s< a_t}+ \overline b\int_s^t \1_{a+\tau-s\leq a_\tau }m_{\tau,t}(0)\,d\tau \, \leq \, \1_{t-s< a_t}
+\overline b\int_s^t \1_{\tau-s\leq a_\tau }m_{\tau,t}(0)\,d\tau\\
&\leq\1_{t-s< a_t}+\frac{\overline b}{\underline b}\int_s^t b_\tau(\tau-s)m_{\tau,t}(0)\,d\tau
\leq\frac{\overline b}{\underline b}m_{s,t}(0).
\end{align*}

Then we compare $m_{s,t}(0)$ to $\nu(m_{s,t}).$
We consider separately the cases $s\leq t\leq s+\alpha$ and $t>s+\alpha.$
For $s\leq t\leq s+\alpha$ we have, since $\alpha+t-s\leq 2\alpha=\Delta\leq a(0)\leq a_t,$
\[\int_0^\alpha m_{s,t}(a)\,da\geq\int_0^\alpha\1_{a+t-s< a_t}da\geq\alpha\]
and we have already seen that $m_{s,t}(0)\leq \e^{\overline b(t-s)}$, so
\[m_{s,t}(0)\leq \e^{\alpha\overline b}\nu(m_{s,t})\]
and
\[\frac{\underline b}{\overline b}\e^{-\alpha\overline b}\|m_{s,t}\|_\infty\leq\nu(m_{s,t}).\]
For the case $t>s+\alpha$ we split into two steps.
We start by comparing $m_{s,t}(0)$ to $m_{r,t}(0)$ for $s\leq r\leq s+\alpha.$
The semigroup property allows to decompose $m_{s,t}(0)=\delta_0 M_{s,r}m_{r,t}.$
Lemma~\ref{lm:agemax:Gronwall} ensures that $\|\delta_0M_{s,r}\|_{\mathrm{TV}}\leq\e^{\overline b(r-s)}$ so,
using again the bound $m_{r,t}(a)\leq\frac{\overline b}{\underline b}m_{r,t}(0),$ we get
\[m_{s,t}(0)=\delta_0 M_{s,r}m_{r,t}\leq \|\delta_0M_{s,r}\|_{\mathrm{TV}}\frac{\overline b}{\underline b}m_{r,t}(0)
\leq\frac{\overline b}{\underline b}\e^{\alpha\overline b}m_{r,t}(0).\]
To conclude we write
\begin{align*}
\frac1\alpha\int_0^\alpha m_{s,t}(a)\,da&\geq\frac1\alpha\int_0^\alpha\int_s^{s+\alpha}b_\tau(a+\tau-s)m_{\tau,t}(0)\,d\tau da\\
&\geq\underline b \int_s^{s+\alpha}m_{\tau,t}(0)\,d\tau \geq\frac{\alpha\underline b^2}{\overline b}\e^{-\alpha\overline b}m_{s,t}(0).
\end{align*}
Finally, we have proved that for all $t>s+\alpha$ we have
\[\frac{\alpha\underline b^3}{\overline b^2}\e^{-\alpha\overline b}\|m_{s,t}\|_\infty\leq\nu(m_{s,t}).\]
Finally, \eqref{as:H3-ah} comes from similar computations and a generalisation of Lemma~\ref{lm:agemax:Gronwall} for $N$.
\end{proof}

\subsection{Periodic semigroups and the renewal equation}
\label{sect:periodic}

In this section, we establish the convergence to a periodic profile for periodic semigroups.
This generalizes the Floquet theory~\cite{F83} for periodic matrices.
We apply this result to the renewal equation and obtain an explicit exponential rate of convergence.
Let us mention that it provides an exponential decay to Floquet eigenelements for a periodic PDE, which up to our knowledge has not been achieved so far.

\subsubsection{Exponential convergence for periodic semigroups} 

We start by a definition of the so-called Floquet eigenelements.

\begin{defi}[Periodic semigroup and Floquet eigenelements]
We say that a semigroup $(M_{s,t})_{0\leq s\leq t}$ is periodic with period $T$ if for all $t\geq0$ we have $\X_{t+T}=\X_t$ and for all
$t\geq s\geq0$,
\[M_{s+T,t+T}=M_{s,t}.\]
We say that $(\lambda_F,\gamma_{s,t},h_{s,t})_{t\geq s\geq0}$ is a Floquet family for $(M_{s,t})_{t\geq s\geq0}$ if
for all $t\geq s\geq0$ the triplet $(\lambda_F,\gamma_{s,t},h_{s,t})\in \R\times\M(\X_s)\times \B_b(\X_t),$
for all $s\geq0$ we have $\gamma_{s,s}\in\P(\X_s)$ and $\gamma_{s,s}(h_{s,s})=1,$
for all $t\geq s\geq0$,
\[\gamma_{s+T,t+T}=\gamma_{s,t}=\gamma_{s,t+T}\qquad\text{and}\qquad h_{s+T,t+T}=h_{s,t}=h_{s,t+T},\]
and
\[\gamma_{s,s}M_{s,t}=\e^{\lambda_F(t-s)}\gamma_{s,t}\qquad\text{and}\qquad M_{s,t}h_{t,t}=\e^{\lambda_F(t-s)}h_{s,t}.\]
\end{defi}

We state the general periodic result, recalling Definition  \eqref{eq:coupling-capacity} of the coupling capacity $C_{\alpha,\beta,\nu}(s,t)$.

\begin{theo}\label{th:Floquet}
Let $(M_{s,t})_{t\geq s\geq0}$ be a $T$-periodic semigroup
and let $\alpha,\beta\geq1$ such that for all $s\geq0$, there exists $\nu\in\P(\X_s)$ such that $C_{\alpha,\beta,\nu}(s,t)\to+\infty$ when $t\to+\infty.$
Assume also that the function $(s,t)\mapsto\|m_{s,t}\|_\infty$ is locally bounded.
Then there exists a unique $T$-periodic Floquet family $(\lambda_F,\gamma_{s,t},h_{s,t})_{t\geq s\geq0}$ for $(M_{s,t})_{t\geq s\geq0}$
and there exist $C,\rho>0$ such that for all $t\geq s\geq0$ and all $\mu\in\M(\X_s)$,
\begin{equation}\label{eq:conv_Floquet}
\big\|\e^{-\lambda_F(t-s)}\mu M_{s,t}-\mu(h_{s,s})\gamma_{s,t}\big\|_{\mathrm{TV}}\leq C\e^{-\rho (t-s)}\left\|\mu\right\|_{\mathrm{TV}}.
\end{equation}
\end{theo}

Notice that in general the exponential rate of convergence $\rho$ can be quantified.

\begin{proof}

We start by the construction of $(\gamma_{s,t})_{t\geq s\geq0}.$
From Theorem~\ref{th:main} there exist $C>0$ and $h_s:\X_s\to(0,\beta]$ such that $\nu(h_s)=1$ and for any $\mu,\gamma\in\M(\X_s)$ and all $t\geq s$,
\be
\label{casgenper}
\biggl\| \mu  M_{s,t}  -  \mu(h_s) \nu(m_{s,t}) \frac{\gamma M_{s,t}}{\gamma(m_{s,t})}  \biggr\|_{\mathrm{TV}} \leq C\, \nu(m_{s,t})\|\mu\|_{\mathrm{TV}}\,\e^{-C_{\alpha,\beta,\nu}(s,t)}.
\ee
Considering $t=s+kT$ for $k\in\N,$ $\mu=\nu$ and $\gamma=\nu M_{s,s+lT}$ for $l\in\N$ and using the periodicity of $M$, we get
\[\biggl\| \frac{\nu M_{s,s+kT}}{\nu(m_{s,s+kT})}  -  \frac{\nu M_{s,s+(k+l)T}}{\nu(m_{s,s+(k+l)T})}  \biggr\|_{\mathrm{TV}} \leq C\, \e^{-C_{\alpha,\beta,\nu}(s,s+kT)},\]
which ensures that $\big(\frac{\nu M_{s,s+kT}}{\nu(m_{s,s+kT})}\big)_{k\in\N}$ is a a Cauchy sequence in $(\M(\X_s),\|\cdot\|_{\mathrm{TV}}).$
We denote the limit $\gamma_{s,s},$ which belongs to $\P(\X_s).$
Using again~\eqref{casgenper} with $\mu=\nu$ we have that for all $\gamma\in\M_+(\X)$,
\be
\label{indsprime}
\frac{\gamma M_{s,s+kT}}{\gamma (m_{s,s+kT})}\xrightarrow[k\to\infty]{}\gamma_{s,s}.
\ee
For $f\in \B_b(\X_s)$ we have, using the periodicity of $M_{s,t},$
\[\gamma_{s,s} M_{s,s+(k+1) T}f=\gamma_{s,s} \, M_{s,s+k T}\,M_{s,s+ T}f,\]
which gives
\begin{equation}\label{eq:gamma_ss}
\frac{\gamma_{s,s} (m_{s,s+(k+1) T})}{\gamma_{s,s}(m_{s,s+k T})}\frac{\gamma_{s,s} M_{s,s+(k+1) T}f}{\gamma_{s,s} (m_{s,s+(k+1) T})}
=\frac{\gamma_{s,s} M_{s,s+kT}}{\gamma_{s,s}(m_{s,s+k T})}M_{s,s+ T}f.
\end{equation}
Letting  $f=\1$ in~\eqref{eq:gamma_ss}, $$\Lambda_s=\lim_{k\to\infty}\frac{\gamma_{s,s}(m_{s,s+(k+1) T})}{\gamma_{s,s} (m_{s,s+k T})}=\gamma_{s,s}(m_{s,s+ T})$$
and letting  $k\to\infty$,
\be
\label{gamlamb}
\Lambda_s\gamma_{s,s} (f)=\gamma_{s,s} M_{s,s+ T}f.
\ee
We check now that $\Lambda_s$ is independent of $s.$
To do so we start by proving that for any $s'>s,$
\be\label{gammas'}\gamma_{s',s'}=\frac{\gamma_{s,s}M_{s,s'}}{\gamma_{s,s}(m_{s,s'})}.\ee
By the semigroup property we have on the one hand, using~\eqref{indsprime} with $\gamma=\gamma_{s,s}M_{s,s'},$
\[\frac{\gamma_{s,s}M_{s,s'+k T}}{\gamma_{s,s}m_{s,s'+k T}}=\frac{(\gamma_{s,s}M_{s,s'})M_{s',s'+k T}}{(\gamma_{s,s}M_{s,s'})m_{s',s'+k T}}\xrightarrow[k\to\infty]{}\gamma_{s',s'},\]
and on the other hand using~\eqref{gamlamb}
\[\frac{\gamma_{s,s}M_{s,s'+k T}}{\gamma_{s,s}m_{s,s'+k T}}=\frac{\gamma_{s,s}M_{s,s+kT}M_{s+k T,s'+k T}}{\gamma_{s,s}M_{s,s+k T}m_{s+k T,s'+k T}}
=\frac{\Lambda_s^k\gamma_{s,s}M_{s,s'}}{\Lambda_s^k\gamma_{s,s}(m_{s,s'})}=\frac{\gamma_{s,s}M_{s,s'}}{\gamma_{s,s}(m_{s,s'})}.\]
This proves~\eqref{gammas'} which gives, using again~\eqref{gamlamb} and the semigroup property,
\[\Lambda_{s'}\gamma_{s',s'}=\gamma_{s',s'}M_{s',s'+T}=\frac{\gamma_{s,s}M_{s,s'+T}}{\gamma_{s,s}(m_{s,s'})}=\frac{\Lambda_s\gamma_{s,s}M_{s,s'}}{\gamma_{s,s}(m_{s,s'})}.\]
Testing this identity against $\1$ and using that $\gamma_{s,s}$ are probabilities we get that $\Lambda_{s'}=\Lambda_s,$ and we denote this constant by $\Lambda.$
Now we define $\lambda_F=(\log\Lambda)/T$ and for all $t\geq s$,
\[\gamma_{s,t}=\gamma_{s,s}M_{s,t}\e^{-\lambda_F(t-s)}.\]
The definition of $\lambda_F$ implies that $\gamma_{s,t}=\gamma_{s,t+T},$ and the identity $\gamma_{s+T,s+T}=\gamma_{s,s},$ which is clear on the definition~\eqref{indsprime} of $\gamma_{s,s},$ ensures the periodicity $\gamma_{s+T,t+T}=\gamma_{s,t}.$

We turn to the family $(h_{s,t})_{t\geq s\geq0}.$
Using remark~\ref{rq:A4} we have that Theorem~\ref{th:main} is valid for $\nu=\gamma_{s,s}$ and it ensures the existence of a harmonic function $$h_{s,s}=\lim_{t\to+\infty}\frac{m_{s,t}}{\gamma_{s,s}(m_{s,t})},$$ which satisfies $\gamma_{s,s}(h_{s,s})=1.$
Now we define for $s\leq t$,
\[h_{s,t}=\e^{-\lambda_F(t-s)}M_{s,t}h_{t,t}.\]
It only remains to check that $(h_{s,t})_{t\geq s\geq0}$ thus defined is $T$-periodic.
By definition of $h_{s,s}$ and using the $T$-periodicity of $M_{s,t}$ and $\gamma_{s,t}$ we have
\[h_{s,s}=\lim_{k\to\infty}\frac{m_{s,s+kT}}{\gamma_{s,s}(m_{s,s+kT})}=\lim_{k\to\infty}\frac{m_{s+T,s+(k+1)T}}{\gamma_{s+T,s+T}(m_{s+T,s+(k+1)T})}=h_{s+T,s+T}.\]
Then we write
\[h_{s,s}=\lim_{t\to\infty}\frac{m_{s,t+T}}{\gamma_{s,s}(m_{s,t+T})}
=\lim_{t\to\infty}\frac{M_{s,s+T}m_{s,t}}{\gamma_{s,s}(M_{s,s+T}m_{s,t})}
=\lim_{t\to\infty}\frac{M_{s,s+T}m_{s,t}}{e^{\lambda_FT}\gamma_{s,s}(m_{s,t})}
=\e^{-\lambda_FT}M_{s,s+T}h_{s,s},\]
where we have used the dominated convergence theorem for the last equality.
And finally
\[h_{s,s+T}=\e^{-\lambda_FT}M_{s,s+T}h_{s+T,s+T}=\e^{-\lambda_FT}M_{s,s+T}h_{s,s}=h_{s,s}.\]

Now we check the convergence~\eqref{eq:conv_Floquet}.
Applying Theorem~\ref{th:main} with $\nu=\gamma=\gamma_{s,s}$ and $s_0=s,$ we get that there exists $C>0$ such that for all $t\geq s\geq 0$ and all $\mu\in\M(\X_s)$,
\[\Bigl\| \mu  M_{s,t}  -  \mu(h_{s,s})\e^{\lambda_F(t-s)} \gamma_{s,t}\Bigr\|_{\mathrm{TV}} \leq C\,\e^{\lambda_F(t-s)} \gamma_{s,t}(\1)\|\mu\|_{\mathrm{TV}}\,\e^{-C_{\alpha,\beta,\nu}(s,t)}.\]
By periodicity we have $\gamma_{s,t}(\1)=\e^{-\lambda_F(t-s)}\gamma_{s,s}(m_{s,t})\leq\e^{|\lambda_F|T}\sup_{0\leq s\leq T, s\leq t\leq s+T}\|m_{s,t}\|_\infty.$
Still by periodicity, for all $n\in\N$ and all $t\geq s+nT$, we have $C_{\alpha,\beta,\nu}(s,t)\geq C_{\alpha,\beta,\nu}(0,nT)\big\lfloor\frac{t-s}{nT}\big\rfloor.$
Using that $C_{\alpha,\beta,\nu}(0,t)\to+\infty$ when $t\to+\infty$ we can choose $n\in\N$ large enough so that $C_{\alpha,\beta,\nu}(0,nT)>0$
and this gives~\eqref{eq:conv_Floquet} with $\rho=C_{\alpha,\beta,\nu}(0,nT)/nT.$

The uniqueness follows from the convergence.
Assume the existence of another Floquet family $(\tilde\lambda,\tilde\gamma_{s,t},\tilde h_{s,t}).$
Applying~\eqref{eq:conv_Floquet} to $\mu=\tilde\gamma_{s,s}$ guarantees that $\tilde\lambda_F=\lambda_F$ and $\tilde\gamma_{s,t}=\tilde\gamma_{s,s}(h_{s,s})\gamma_{s,t}.$
Since $\gamma_{s,s}$ and $\tilde\gamma_{s,s}$ are both probabilities this implies that $\tilde\gamma_{s,s}(h_{s,s})=1,$ and $\tilde\gamma_{s,t}=\gamma_{s,t}.$
Applying again~\eqref{eq:conv_Floquet} but with $\mu=\delta_x,$ $t=s+kT,$ and the test function $\tilde h_{s,s}=\tilde h_{s+kT,s+kT},$ we get $\tilde h_{s,s}=h_{s,s}.$
The equality $\tilde h_{s,t}=h_{s,t}$ is then ensured by the identity $\tilde\lambda_F=\lambda_F.$
\end{proof}

\subsubsection{Example: the periodic renewal equation}

We consider the renewal equation with a time-periodic division rate,
which is used as a model for circadian rhytms (see~\cite{CMP,CGL} and the references therein).
More precisely the equation is
\begin{equation}\label{eq:renewal-periodic}\left\{\begin{array}{ll}
\partial_t u_{s,t}(a)+\partial_a u_{s,t}(a)+b(t,a)u_{s,t}(a)=0, &\qquad  t>s,\ a>0,
\vspace{2mm}\\
\dis u_{s,t}(0)=2\int_0^\infty b(t,a)u_{s,t}(a)\,da, & \qquad  t>s,
\vspace{2mm}\\
u_{s,s}(a)=u_s(a),&\qquad a\geq0,
\end{array}\right.\end{equation}
and we assume that there exists $T>0$ such that $b(t+T,\cdot)=b(t,\cdot)$ for any $t\geq0.$
Additionally $b$ is supposed to be non-negative and globally bounded (in time and age) by a constant $\overline b>0.$
Similarly as in the homogeneous or asymptotically homogeneous case we associate to Equation~\eqref{eq:renewal-periodic} a semigroup $(M_{s,t})_{0\leq s\leq t}$ defined on $\B_b(\R_+)$ by the Duhamel formula
\bea
\label{eq:renewal-periodic:Duhamel}
M_{s,t}f(a)&=&f(a+t-s)\e^{-\int_s^t b(\tau,a+\tau-s)d\tau} \\
&&\qquad +2\int_s^t\e^{-\int_s^\tau b(\tau',a+\tau'-s)d\tau'}b(\tau,a+\tau-s)M_{\tau,t}f(0)\,d\tau \nonumber
\eea
and on $\M(\R_+)$ by setting $(\mu M_{s,t})(f)=\mu(M_{s,t}f)$ for all $f\in\B_b(\R_+).$
As for the homogeneous case, this semigroup is well defined and satisfies Assumption~\ref{as:Mst}.
Mimicking the proof of Lemma~\ref{lm:renewal:mt} we can prove the following monotonicity result.

\begin{lem}\label{lm:renewal-periodic:monotone}
For all $s,a\geq0,$ the function $t\mapsto m_{s,t}(a)$ is nondecreasing.
\end{lem}

Notice however that the function $s\mapsto m_{s,t}(a)$ is not nonincreasing in general.
Yet, since $(M_{s,t})_{0\leq s\leq t}$ is $T$-periodic we have, using Lemma~\ref{lm:renewal-periodic:monotone}, $m_{s+T,t}(a)=m_{s,t-T}(a)\leq m_{s,t}(a).$

The aim of the current section is to provide sufficient conditions on $b$ so that we can apply Theorem~\ref{th:Floquet}.
In Theorem~\ref{th:renewal-periodic:general} we give a general such result,
and in Theorem~\ref{th:renewal-periodic:constant} we optimize the rate of convergence in the case when the division rate depends only on time.
Let us point out that in this latter case the mean number of individuals $m_{s,t}(a)$ does not depend on $a$ and satisfies a simple differential equation.
Classical Floquet theory \cite{F83} then easily applies for proving that it tends to a periodic solution.
However it is no longer the case for the age repartition described by the semigroup which remains an infinite-dimensional object.
Its asymptotic periodicity cannot be deduced from those of $m_{s,t}$.

The convergence of the solutions of the periodic renewal equation to the Floquet elements has been obtained in~\cite{MMP} by the way of entropy techniques. Here, we provide an explicit exponential rate of convergence.

\begin{theo}\label{th:renewal-periodic:general}
Let $b$ be a time-periodic function with period $T>0,$ non-negative and globally bounded by $\overline b>0.$
Assume that there exist $A\geq0$ and $\underline b>0$ such that
\[\forall t\geq0,\ \forall a\geq A,\qquad b(t,a)\geq\underline b.\]
Then there exists a unique Floquet family $(\lambda_F,\gamma_{s,t},h_{s,t})_{0\leq s\leq t}$ for the semigroup $(M_{s,t})_{0\leq s\leq t}$ and a constant $C>0$ such that for all $t\geq s\geq0$ and all $\mu\in\M(\R_+)$,
\[\big\|\e^{-\lambda_F(t-s)}\mu M_{s,t}-\mu(h_{s,s})\gamma_{s,t}\big\|_{\mathrm{TV}}\leq C\,\|\mu\|_{\mathrm{TV}}\,\e^{-\rho (t-s)}\]
where
\[\rho=\frac{-1}{A+2T}\log\bigg(1-\frac{2\underline bT\e^{-\overline b(3A+8T)}}{\frac{1}{2\overline bT}+\frac{A}{T}+3+\frac{1}{1-\e^{-\underline bT}}}\bigg).\]
\end{theo}

\begin{proof}
We will use several times the inequality
\[\|m_{s,t}\|_\infty\leq\e^{2\overline b(t-s)},\]
whose proof follows Lemma~\ref{lm:agemax:Gronwall}.
In particular it ensures that the function $(s,t)\mapsto\|m_{s,t}\|_\infty$ is locally bounded.

Now we exhibit constants for the assumptions in Definition~\ref{def:coupling_constants}. We start with \eqref{as:A1I}.
Let $s\geq0$ and define $n:=\big\lfloor\frac{A}{T}\big\rfloor,$ so that $(n+1)T\in(A,A+T].$
From the Duhamel formula~\eqref{eq:renewal-periodic:Duhamel} and using the periodicity of the semigroup we get that for any $f\geq0$ and any $a\geq0$
\begin{align*}
M_{s,s+(n+2)T}f(a)&\geq2\int_{s+(n+1)T}^{s+(n+2)T}\e^{-\int_s^\tau b(\tau',a+\tau'-s)d\tau'}b(\tau,a+\tau-s)M_{\tau,s+(n+1)T}f(0)\,d\tau\\
&\geq 2\underline b\e^{-\overline b(A+2T)}\int_{s}^{s+T}M_{\tau,s+T}f(0)\,d\tau.
\end{align*}
We have $m_{s,s+(n+2)T}\leq\e^{2\overline b(A+2T)}$ and Lemma~\ref{lm:renewal-periodic:monotone} ensures that $\int_s^{s+T}m_{\tau,s+T}(0)\,d\tau\geq T.$
So \eqref{as:A1I} is satisfied with
$$\nu(f)=\frac{\int_s^{s+T}M_{\tau,s+T}f(0)\,d\tau}{\int_s^{s+T}m_{\tau,s+T}(0)\,d\tau}\qquad \text{ and } \qquad
c=2\underline bT\e^{-3\overline b(A+2T)}.$$
We treat the last three assumptions \eqref{as:A2I}-\eqref{as:A4I} together by proving that there exists $d>0$ such that for all $t\geq s$,
\[d\,\|m_{s,t}\|_\infty\leq \nu(m_{s,t}).\]
If $t-s\leq A$, we have $\|m_{s,t}\|_\infty\leq \e^{2\overline b A},$
and from Lemma~\ref{lm:renewal-periodic:monotone}
$\nu_s(m_{s,t})\geq1,$
so it remains to treat the case $t-s>A.$
Keeping $n=\big\lfloor\frac{A}{T}\big\rfloor$ and setting $N=\big\lfloor\frac{t-s}{T}\big\rfloor\geq n$, we have from~\eqref{eq:renewal-periodic:Duhamel}
\begin{align*}
m_{s,t}(a)&=\e^{-\int_s^tb(\tau,a+\tau-s)d\tau}+2\int_s^t\e^{-\int_s^\tau b(\tau',a+\tau'-s)d\tau'}b(\tau,a+\tau-s)m_{\tau,t}(0)\,d\tau\\
&\leq1+2\overline b\sum_{k=0}^{n}\int_{s+kT}^{s+(k+1)T}\e^{-\int_s^\tau b(\tau',a+\tau'-s)d\tau'}m_{\tau,t}(0)\,d\tau\\
&\hspace{10mm}+2\overline b\sum_{k=n+1}^{N}\int_{s+kT}^{s+(k+1)T}\e^{-\int_s^\tau b(\tau',a+\tau'-s)d\tau'}m_{\tau,t}(0)\,d\tau\\
&\hspace{20mm}+2\overline b\int_{s+NT}^t\e^{-\int_s^\tau b(\tau',a+\tau'-s)d\tau'}m_{\tau,t}(0)\,d\tau\\
&\leq1+2\overline b\sum_{k=0}^{n}\int_{s}^{s+T}m_{\tau+kT,t}(0)\,d\tau\\
&\hspace{10mm}+2\overline b\sum_{k=n+1}^{N}\int_{s}^{s+T}\e^{-\int_s^{\tau+kT} b(\tau',a+\tau'-s)d\tau'}m_{\tau+kT,t}(0)\,d\tau\\
&\hspace{20mm}+2\overline b\int_{s}^{t-NT}m_{\tau+NT,t}(0)\,d\tau\\
&\leq1+2\overline b(n+1)\int_{s}^{s+T}m_{\tau,t}(0)\,d\tau+2\overline b\int_{s}^{t-NT}m_{\tau,t}(0)\,d\tau\\
&\hspace{10mm}+2\overline b\sum_{k=n+1}^{N}\int_{s}^{s+T}\e^{-\int_{(n+1)T}^{kT} \underline b\1_{t_1\leq u-\lfloor u/T\rfloor\leq t_2}\,du}m_{\tau,t}(0)\,d\tau\\
&\leq(1/T+2\overline b(n+2))\int_{s}^{s+T}m_{\tau,t}(0)\,d\tau+2\overline b\sum_{k=0}^{N-n-1}\e^{-k\underline b(t_2-t_1)}\int_{s}^{s+T}m_{\tau,t}(0)\,d\tau\\
&\leq\bigg(\frac1T+2\Big(\frac AT+3\Big)\overline b+\frac{2\overline b}{1-\e^{-\underline bT}}\bigg)\int_{s}^{s+T}m_{\tau,t}(0)\,d\tau.\\
\end{align*}
On the other hand we have
\begin{align*}
\nu(m_{s,t})&=\nu(m_{s+T,t+T})=\frac{1}{\int_s^{s+T}m_{\tau,s+T}(0)\,d\tau}\int_s^{s+T}m_{\tau,t+T}(0)\,d\tau\\
&\geq\frac{1}{\int_s^{s+T}\e^{2\overline b(s+T-\tau)}\,d\tau}\int_s^{s+T}m_{\tau,t}(0)\,d\tau\geq 2\overline b\e^{-2\overline bT}\int_s^{s+T}m_{\tau,t}(0)\,d\tau,
\end{align*}
which  gives
\[m_{s,t}(a)\leq\bigg(\frac1{2\overline bT}+\frac AT+3+\frac{1}{1-\e^{-\underline bT}}\bigg)\e^{2\overline bT}\nu_s(m_{s,t})\]
and ends the proof.
\end{proof}

\begin{theo}\label{th:renewal-periodic:constant}
Assume that $b(t,a)=b(t)$ is a continuous $T$-periodic function, which is not identically zero.
Then there exists a unique Floquet family $(\lambda_F,\gamma_{s,t},h_{s,t})_{0\leq s\leq t}$ for the semigroup $(M_{s,t})_{0\leq s\leq t}$ and a constant $C>0$ such that for all $t\geq s\geq0$ and all $\mu\in\M(\R_+)$
\[\big\|\e^{-\lambda_F(t-s)}\mu M_{s,t}-\mu(h_{s,s})\gamma_{s,t}\big\|_{\mathrm{TV}}\leq C\,\|\mu\|_{\mathrm{TV}}\,\e^{-2\int_s^tb(\tau)d\tau}.\]
\end{theo}

\begin{proof}
Since $b$ does not depend on $a$ we have the explicit formula
\begin{equation}
\label{eq:yuleperiodiq}
m_{s,t}(a)=\e^{\int_s^t b(\tau)d\tau}.
\end{equation}
In particular it ensures that the function $(s,t)\mapsto\|m_{s,t}\|_\infty$ is locally bounded.

Now we prove \eqref{as:A1I}.
Let $t\geq s+T,$ $k\in\N,$ $n=\big\lfloor\frac{t-s}{T}\big\rfloor\geq1,$ and $N=(n-1)k+1.$
For all $0\leq i\leq N-1$ we set $t_i=s+i\frac{T}{k},$ and $t_N=s+nT.$
From the Duhamel formula~\eqref{eq:renewal-periodic:Duhamel} we have for any $1\leq i\leq N-1$ and $f\geq0$
\begin{align*}
M_{t_{i-1},t_i}f(a)
&\geq2\int_{t_{i-1}}^{t_i}\e^{-\int_{t_{i-1}}^\tau b(\tau')d\tau'}b(\tau)f(t_i-\tau)\,d\tau\geq2\e^{-\|b\|_\infty\frac Tk}\Big(\min_{[t_{i-1},t_i]}b\Big)\int_0^{\frac{T}{k}}f(\tau)\,d\tau.
\end{align*}
For $i=N$ we write
\begin{align*}
M_{t_{N-1},t_N}f(a)
&\geq2\int_{s+(n-1)T}^{s+nT}\e^{-\int_{s+(n-1)T}^\tau b(\tau')d\tau'}b(\tau)f(s+(n-1)T-\tau)\,d\tau\\
&\geq2\e^{-\|b\|_\infty T}\int_0^Tb(T-\tau)f(\tau)\,d\tau.
\end{align*}
So \eqref{as:A1I} is verified for $0\leq i\leq N-1$ with $\nu_i(f)=\frac{n}{T}\int_0^{\frac{T}{n}}f(\tau)\,d\tau$ and
\[c_i=\frac{2T}{k}\e^{-\|b\|_\infty\frac Tk}\min_{[t_{i-1},t_i]}b\]
and for $i=N$ with $\nu_N(f)=\frac{\int_0^Tb(T-\tau)f(\tau)\,d\tau}{\int_0^Tb(\tau)\,d\tau}$ and
\[c_N=2\e^{-\|b\|_\infty T}\int_0^Tb(\tau)\,d\tau.\]

From~\eqref{eq:yuleperiodiq} we deduce that Assumptions \eqref{as:A2I}-\eqref{as:A4I} are trivially verified with
\[d_i=1,\qquad \alpha=\frac1{c_N},\qquad\beta=1.\]

Thus the coupling capacity satisfies, for all $k\in\N,$
\[C_{\alpha,\beta,\nu}(s,t)\geq-\sum_{i=1}^{N-1}\log\bigg(1-\frac{2T}{k}\e^{-\|b\|_\infty\frac Tk}\min_{[t_{i-1},t_i]}b\bigg).\]
Letting $k\to\infty$ we get
\[C_{\alpha,\beta,\nu}(s,t)\geq2\int_{s}^{s+(n-1)T}b(\tau)\,d\tau\geq2\int_{s}^{t}b(\tau)\,d\tau-2\int_0^Tb(\tau)\,d\tau.\]

\end{proof}

\begin{Rq}[About optimality]
In the constant case $b=const>0$, we recover the spectral gap $2b.$  Indeed,
we know that the spectral gap of the operator
\[\A f(a)=f'(a)-bf(a)+2bf(0)\]
cannot be larger than $2b$ because the dominant Perron eigenvalue is $b,$ and $-b$ belongs to the spectrum of $\A$
(the operator $\A+b$ is not surjective on $C_b(\R_+)$ since all the solutions to the equation $f'(a)=-2bf(0)+\frac{1}{1+a}$ for instance are unbounded).
\end{Rq}

\begin{Rq}[Eigenvalue]
From the periodicity of $(\gamma_{s,t})_{t\geq s\geq0}$ and \eqref{eq:yuleperiodiq}, we easily get that
$$
\lambda_F = \frac{1}{T} \int_0^T b(\tau) d\tau.
$$
\end{Rq}

\appendix

\section{Branching models, absorbed Markov process and semigroups}
\label{appendix:proba}
The techniques by coupling used in this paper have been extensively developed in probability, in particular for the study 
of branching processes and killed process, see the introduction for references. Let us present here informally the probabilistic 
objects and the interpretation of the auxiliary semigroup. \\
For that purpose we consider a population of individuals with a trait belonging to the space $\X$. This population can die or give birth to some offsprings 
with a rate which depends on their trait and independently one from each other (branching property). Moreover the trait may vary in an homogeneous
way and without memory (Markov property). Let us also assume that some subspace $\mathcal S$ of $\X$ is absorbing, meaning that each individual whose trait reaches this set stop dividing and keeps a constant trait. Writing $V_t$ the set of individuals at time $t$ and 
$(X_t^i :  i \in V_t)$ the set of their traits, the branching and Markov properties and the absorbing property of $\mathcal S$ ensure that
$$\delta_x M_{s,t}(f)=\E\left( \sum_{i\in V_t} f(X_t^i)1_{X_t^i \not\in \mathcal S} \  \big\vert \ X_s=\delta_x \right)$$
is a semigroup. In general, it is not conservative, since its mass 
$$m_{s,t}(x)=\delta_x M_{s,t}{\bf 1}=\E(\#\{ i \in V_t : X_t^i \not\in \mathcal S  \  \big\vert \ X_s=\delta_x\})$$
can decrease  by absorption in $\mathcal S$ or death of individual or created by births. The trait of a typical non-absorbed individual
is then given by the auxiliary conservative inhomogeneous semigroup
$$\delta_xP^{(t)}_{u,v}f=\frac{\delta_x M_{u,v}(fm_{v,t})}{m_{u,t}(x)}=\frac{\E\left( \sum_{i\in V_t} f(X_v^i)1_{X_v^i \not\in \mathcal S} \big\vert X_u=\delta_x\right)}{\E(\#\{ i \in V_t : X_t^i \not\in \mathcal S\} \big \vert X_u=\delta_x )}=\E(f(Y_v^{(t)}) \vert Y_u^{(t)}=x),$$
where $X_v^i$ is the trait of the ancestor of $i$ at time $v$ and $Y^{(t)}$
is the inhomogenous Markov process  associated to $P^{(t)}$. Thus, $Y^{(t)}$ is the process describing the dynamics of the trait of a typical individual, which is alive at time $t$ and non-absorbed. 
Proving that it is ergodic ensures the ergodicity of $\delta_x M_{s,t}{\bf 1}/m_{s,t}(x)$ as $t$ goes to infinity. In this paper, we  make a coupling for that, with Doeblin conditions which ensure exponential uniform ergodicity. Thanks to \cite{CV16}, this Doeblin condition can be rewritten in terms of   coupling constants on the original semigroup $M$. \\

In homogeneous-time setting, two particular classes of processes have attracted lots of attention. First, if  we make $\mathcal S=\varnothing$, then $X$ is a branching process and 
$$\delta_x M_{s,t}(f)=\E\left( \sum_{i\in V_t} f(X_t^i) \  \big\vert \ X_s=\delta_x \right)$$
is its first moment semigroup which provides the mean number of individuals with a given trait.
The auxiliary process describes the dynamical of the trait along the ancestral lineage of an individual chosen uniformly at random, when the population is becoming large. More generally, the genealogical tree of the population can be constructed  from this typical lineage, which is called \emph{spine construction}. \\
Second, if the individuals neither die nor give birth, we get a Markov process in the space trait $\X$ and 
$$\delta_x M_{s,t}(f)=\E(f(X_t)1_{X_t\not\in \mathcal S} \ \vert \ X_s=x)$$
Assume that $X_t$ is eventually absorbed as $t$ goes to infinity a.s. and consider the distribution of the process conditioned on non-absorption :
$$\mathbb P_x(X_t \in .  \ \vert \  X_t \not\in \mathcal S)=\frac{\delta_x M_{0,t}}{m_{0,t}(x)}=\delta_xP^{(t)}_{0,t}.$$
The ergodic behavior of $P^{(t)}$ and its convergence to a distribution $\nu$ yields the convergence of the conditioned distribution (Yaglom limit) to the quasistationary distribution. At fixed time $t$, $P^{(t)}$ describes the dynamic of the trait for trajectories non-absorbed at time $t$.

\section{Measure solutions to the renewal equation}\label{appendix:renewal}

We give here the details about the construction of the homogeneous renewal semigroup.
It is based on the dual renewal equation
\begin{equation}\label{eq:renewal:dual}
\hskip8mm\partial_t f_t(a)-\partial_a f_t(a)+b(a)f_t(a)=2b(a)f_t(0), \hskip14mm t,a\geq0.
\end{equation}
Integrating this equation along the characteristics, we obtain the mild formulation (also called Duhamel formula)
\begin{equation}\label{eq:renewal:Duhamel2}
f_t(a)=f_0(a+t)\e^{-\int_0^tb(a+\tau)d\tau}+2\int_0^t\e^{-\int_0^\tau b(a+\tau')d\tau'}b(a+\tau)f_{t-\tau}(0)\,d\tau.
\end{equation}
The first result is about the well-posedness of this equation in $\B_b(\R_+).$

\begin{lem}\label{prop:renewal:Duhamel}
For all $f_0\in \B_b(\R_+)$ there exists a unique family $(f_t)_{t\geq0}\subset \B_b(\R_+)$ solution to~\eqref{eq:renewal:Duhamel2}.
Additionally if $f_0\geq0$ then $f_t\geq0$ for all $t\geq0.$
\end{lem}

\begin{proof}
First we use the Banach fixed point theorem on a truncated problem.
For $T>0$ and $f_0\in \B_b(\R_+)$ we define the operator $\Gamma:\B_b([0,T])\to \B_b([0,T])$ by
\[\Gamma g(t)= f_0(t)\e^{-\int_0^tb(\tau)d\tau}+2\int_0^t \Phi(\tau)g(t-\tau)\,d\tau.\]
We easily have
\[\|\Gamma g_1-\Gamma g_2\|_\infty\leq 2\int_0^T\Phi(\tau)\,d\tau\,\|g_1-g_2\|_\infty,\]
so $\Gamma$ is a contraction if $2\int_0^T\Phi<1$ and there is a unique fixed point in $\B_b([0,T]).$
Additionally since $\Gamma$ preserves non-negativity when $f_0\geq0,$
we get that the fixed point $g$ is non-negative when $f_0$ is non-negative.
Since the contraction constant $2\int_0^T\Phi$ is independent of $f_0,$ we can iterate to obtain a unique function $g\in \B_b(\R_+)$ which satisfies
\[g(t)= f_0(t)\e^{-\int_0^tb(\tau)d\tau}+2\int_0^t \Phi(\tau)g(t-\tau)\,d\tau\]
for all $t\geq0.$
Now we set for all $t,a\geq0$
\[f_t(a)=f_0(a+t)\e^{-\int_0^tb(a+\tau)d\tau}+2\int_0^t\e^{-\int_0^\tau b(a+\tau')d\tau'}b(a+\tau)g(t-\tau)\,d\tau,\]
which is a solution to~\eqref{eq:renewal:Duhamel} since by definition $f_t(0)=\Gamma g(t)=g(t).$
The uniqueness is a direct consequence of the uniqueness of $g.$
The non-negativity follows from the non-negativity of $g$ when $f_0\geq0,$
and the boundedness is given by the inequality
\[\|f_t\|_\infty\leq\|f_0\|_\infty+2\sup_{0\leq s\leq t}|g(s)|.\]
\end{proof}

Lemma~\ref{prop:renewal:Duhamel} allows to define for all $t\geq0$ the operator $M_t$ on $\B_b(\R_+)$ by setting $M_tf_0:=f_t$ for all $f_0\in \B_b(\R_+).$
Then for $\mu\in\M_+(\R_+)$ we define the positive measure $\mu M_t$ by setting for all Borel set $A\subset\R_+$
\[(\mu M_t)(A):=\mu(M_t\1_A).\]
The axioms of a positive measure are satisfied.
First it is clear that $(\mu M_t)(\varnothing)=0$ and that $(\mu M_t)(A\cup B)=(\mu M_t)(A)+(\mu M_t)(B)$ when $A$ and $B$ are two disjoint Borel sets.
The last axiom deserves a bit more attention.
Let $(A_n)_{n\geq0}$ be an increasing sequence of Borel sets and $A=\bigcup_{n\geq0}A_n.$
We want to check that $(\mu M_t)(A)=\lim_{n\to\infty}(\mu M_t)(A_n).$
The sequence $(\1_{A_n})_{n\geq0}$ is an increasing sequence of Borel functions which converges pointwise to $\1_A.$
By positivity of the semigroup, $(M_t\1_{A_n})_{n\geq0}$ is an increasing sequence of Borel functions bounded by $M_t\1.$
Thus this sequence admits a pointwise limit $f_t\in\B_b(\R_+)$ which clearly satisfies the Duhamel formula~\eqref{eq:renewal:Duhamel} with $f_0=\1_A.$
By uniqueness of the solution to the Duhamel formula we get that $M_t\1_{A_n}\to M_t\1_A$ pointwise.
Then by dominated or monotone convergence we deduce that $(\mu M_t)(A)=\mu(M_t\1_A)=\lim_{n\to\infty}\mu(M_t\1_{A_n})=\lim_{n\to\infty}(\mu M_t)(A_n).$
Finally for a signed measure $\mu\in\M(\R_+)$ we set obviously $\mu M_t:=\mu_+ M_t-\mu_- M_t.$

\medskip

The family $(M_t)_{t\geq0}$ such defined is a semigroup which satisfies Assumption~\ref{as:Mst}.
The semigroup property is a consequence of the uniqueness of the solution to the Duhamel formula~\eqref{eq:renewal:Duhamel}.
The positivity has been proved in Lemma~\ref{prop:renewal:Duhamel}.
For the strong positivity it follows from the Duhamel formula~\eqref{eq:renewal:Duhamel} that for all $t,a\geq0$
\[m_t(a)\geq\e^{-\int_0^t b(a+\tau)d\tau}>0.\]
The compatibility $(\mu M_t)(f)=\mu(M_tf)$ follows directly from the definition of $\mu M_t$ and the definition of Borel functions.

\medskip

It is claimed in Section~\ref{sect:hom-renewal} that the family $(\mu M_t)_{t\geq0}$ is a measure solution to the renewal equation.
Measure valued solutions to structured population models drew interest in the last few years~\cite{GwiazdaWiedemann,CanizoCarrilloCuadrado,CarrilloColomboGwiazdaUlikowska,GwiazdaLorenzMarciniak,G17}.
They are mainly motivated by biological applications which often require to consider initial distributions which are not densities but measures (Dirac masses for instance).
For us it is additionally the suitable framework to apply our ergodic result in Theorem~\ref{th:homogene}.
We refer to~\cite{G17} for the proof that the family $(\mu M_t)_{t\geq0}$ is a measure valued solution to Equation~\eqref{eq:renewal} for any $\mu\in\M(\R_+)$.
Here we only give a heuristic argument which consists in differentiating the semigroup property $\mu M_t f=\mu M_s M_{t-s}f$ with respect to $s\in[0,t].$
The chain rule gives
\[\partial_s(\mu M_s)M_{t-s}f+\mu M_s\partial_s(M_{t-s}f)=0\]
and since $M_tf$ is a solution to the dual Equation~\eqref{eq:renewal:dual} this gives
\[\partial_s(\mu M_s)M_{t-s}f-\mu M_s\A (M_{t-s}f)=0\]
where $\A$ is the unbounded operator defined on $C^1(\R_+)$ by $\A f(a)=f'(a)-b(a)f(a)+2b(a)f(0).$
Taking $s=t$ we get that for all bounded and continuously differentiable function $f$
\[\partial_t(\mu M_tf)=\mu M_t(\A f),\]
which is a weak formulation of Equation~\eqref{eq:renewal}.

\

\section{The max-age semigroup}\label{appendix:maxage}

As for the homogeneous renewal equation, to build a solution to Equation~\eqref{eq:age_max} we use a duality approach.
We start with the (backward) dual equation
\be\label{eq:dual_age_max}\left\{\begin{array}{ll}
\partial_s f_{s,t}(a)+\partial_af_{s,t}(a)+b(a)f_{s,t}(0)=0,\qquad\qquad& s<t,\ 0\leq a< a_s,
\vspace{3mm}\\
f_{s,t}(a_s)=0,& s<t,
\vspace{3mm}\\
f_{t,t}(a)=f_t(a),& 0\leq a< a_t.
\end{array}\right.\ee
Integrating this equation along the characteristics we get the Duhamel formula
\be\label{eq:Duhamel_age_max2}
f_{s,t}(a)=f_t(a+t-s)+\int_s^t b_\tau(a+\tau-s)f_{\tau,t}(0)\,d\tau
\ee
where we have denoted $b_t(a):=b(a)\1_{[0,a_t)}(a)$ and $f_t$ has been extended by 0 beyond $a_t.$

\begin{lem}\label{prop:agemax:Duhamel}
For all $t>0,$ $f_t\in \B_b\big([0,a_t)\big),$ and $s\in[0,t],$ there exists a unique $f_{s,t}\in \B_b\big([0,a_s)\big)$ which satisfies~\eqref{eq:Duhamel_age_max2}.
Additionally if $f_t\geq0$ then $f_{s,t}\geq0.$
\end{lem}

We do not repeat the proof of this result since it follows exactly the strategy of the proof of Lemma~\ref{prop:renewal:Duhamel}.
As for the homogeneous renewal equation we define the semigroup $(M_{s,t})_{0\leq s\leq t}$ on $(\X_t)_{t\geq0}=\big([0,a_t)\big)_{t\geq0},$
first on the right hand side by setting for all $f_t\in \B_b([0,a_t))$
\[M_{s,t}f_t:=f_{s,t}\]
where $f_{s,t}$ is the unique solution to Equation~\eqref{eq:Duhamel_age_max2},
and then on the left by setting for all $\mu\in\M([0,a_s))$ and all Borel set $A\subset[0,a_t)$
\[(\mu M_{s,t})(A)=\mu_+(M_{s,t}\1_A)-\mu_-(M_{s,t}\1_A).\]
For any $\mu\in\M\big([0,a_s)\big)$ the family $(\mu M_{s,t})_{s\leq t}$ is a measure solution to Equation~\eqref{eq:age_max}.
As for the homogeneous case a non rigorous justification is obtained by differentiating the semigroup property $\mu M_{s,t}f=\mu M_{s,r}M_{r,t}f$ with respect to $r\in[s,t]$ and using that $M_{r,t}f$ is solution to~\eqref{eq:dual_age_max}.

\medskip

The semigroup property for the family $(M_{s,t})_{t\geq s\geq 0}$ is a consequence of the uniqueness of the solution to the Duhamel formula~\eqref{eq:Duhamel_age_max2}.
We now verify Assumption~\ref{as:Mst}.
The positivity has been proved in Lemma~\ref{prop:agemax:Duhamel}.
For the strong positivity it suffices to check that $m_{s,t}(0)>0.$
Indeed if $m_{s,t}(0)>0$ for all $0\leq s\leq t$ the Duhamel formula ensures that for $a<a_s$
\[\qquad m_{s,t}(a)=\1_{a+t-s< a_t}+\int_s^t b_\tau(a+\tau-s)m_{\tau,t}(0)\,d\tau\geq \underline b \int_s^t \1_{a+\tau-s< a_\tau }\,m_{\tau,t}(0)\,d\tau>0.\]
The positivity of $m_{s,t}(0)$ is clear if $t-s<a_t$ since
\[m_{s,t}(0)\geq\1_{t-s< a_t}.\] 
Consider  now the case $t-s\geq a_t.$
The function $r\mapsto m_{r,t}(0)$ is continuous on $[s,t-a_t]$ since for $r\leq t-a_t$ we have
\[m_{r,t}(0)=\int_r^t b_\tau(\tau-r)m_{\tau,t}(0)\,d\tau.\]
Assume by contradiction that there exists $r_0\in[s,t-a_t]$ such that $m_{r_0,t}(0)=0$ and $m_{r,t}(0)>0$ for all $r\in(r_0,t].$
Then the equality above would give for $r=r_0$
\[0=\int_{r_0}^t b_\tau(\tau-r_0)m_{\tau,t}(0)\,d\tau,\]
which is not possible since the integrand on the right hand side is positive for $\tau$ close to~$r_0.$
Finally the compatibility condition readily follows from the definition of $\mu M_{s,t}.$

\section*{Acknowledgments}

B.C. and V.B. have received the support of the Chair ``Mod\'elisation Math\'ematique et Biodiversit\'e'' of VEOLIA-Ecole Polytechni\-que-MnHn-FX.
The three authors have been supported by ANR projects, funded by the French Ministry of Research:
B.C. by ANR PIECE (ANR-12-JS01-0006-01),
V.B. by ANR ABIM (ANR-16-CE40-0001) and ANR CADENCE (ANR-16-CE32-0007),
and P.G. by ANR KIBORD (ANR-13-BS01-0004).


\end{document}